\documentclass[11pt]{article}
\usepackage[utf8]{inputenc}
\usepackage{mathtools,amsmath,amssymb}
\usepackage{graphicx}
\usepackage{lmodern}
\usepackage{}
\usepackage[T1]{fontenc}
\usepackage[pagebackref=false]{hyperref}
\usepackage{bbm}
\usepackage{bm}
\usepackage{appendix}
\usepackage{comment}
\usepackage{mathtools}
\usepackage{array}
\usepackage{amsthm}
\usepackage{authblk}
\usepackage[left=2cm,right=2cm,top=2cm,bottom=2cm]{geometry}
\numberwithin{equation}{section}

\newtheorem{theorem}{Theorem}

\newtheorem{lemma}{Lemma}
\newtheorem{corollary}{Corollary}
\theoremstyle{definition}
\newtheorem{remark}{Remark}

\DeclareMathOperator*\uplim{\overline{lim}}
\DeclareMathOperator*\lowlim{\underline{lim}}

\newcommand{\calI}{\mathcal{I}}
\newcommand{\bfrho}{\boldsymbol{\rho}}
\newcommand{\bfG}{\boldsymbol{G}}
\newcommand{\bfH}{\boldsymbol{H}}

\newcommand{\E}{\mathbb{E}}
\newcommand{\p}{\mathbb{P}}
\newcommand{\dd}{\mathrm{d}}

\usepackage{xspace}

\usepackage{graphicx}

\usepackage{subcaption} 

 \allowdisplaybreaks
\usepackage{tikz}
\usetikzlibrary{decorations.pathreplacing}
\title{A large deviation principle for the multispecies stirring process}
\author[1]{Francesco Casini}
\author[2]{Frank Redig}
\author[2]{Hidde van Wiechen}
\affil[1]{\small{Università di Modena e Reggio Emilia, FIM, Modena, Italy}}
\affil[2]{Delft Institute of Applied Mathematics, TU Delft, Delft, The Netherlands}
\date{\today}

\begin{document}

\maketitle

\begin{center}
    \section*{Abstract}
    \end{center}
\begin{changemargin}{20mm}{20mm}
In this paper we consider the multispecies stirring process on the discrete torus. We prove a large deviation principle for the trajectory of the vector of densities of the different species. The technique of proof consists in extending the method of the foundational paper \cite{KOV} based on the superexponential estimate to the multispecies setting. This requires a careful choice of the corresponding weakly asymmetric dynamics, which is parametrized by fields depending on the various species. We also prove the hydrodynamic limit of this weakly asymmetric dynamics, which is similar to but different from the ABC model in \cite{ABC-hd}.
Using the appropriate asymmetric dynamics, we also obtain that the mobility matrix relating the drift currents to the fields coincides with the covariance matrix of the reversible multinomial distribution, which then further leads to the Einstein relation.

\emph{Keywords}: multispecies, large deviation principle, hydrodynamic limit, weakly asymmetric model, exponential martingale, Girsanov formula, exponential tightness.

\end{changemargin}
\newpage

\section{Introduction}
\subsection{Motivations}
Interacting particle systems \cite{spitzer,liggett1985} are used to study how macroscopic equations emerge from microscopic stochastic dynamics, as well as in the study of driven non-equilibrium systems, and their non-equilibrium steady states.
Among these, a well-studied process is the so called Simple Symmetric Exclusion Process (SSEP), where particle interactions are governed by an exclusion constraint that permits at most one particle per site. This model (and various modifications of it) has been extensively studied in the literature, both in the study of scaling limits \cite{KipnisLandim,demasiPresutti,seppalainen, spohn2012large} as well as in the understanding of microscopic properties of non-equilibrium steady states \cite{schutz1994, giardina2009duality} . The study of large deviations of the trajectory of the empirical density field for the SSEP was initiated in \cite{KOV} (see also \cite{guo1988nonlinear} where the gradient method was introduced in the context of Ginzburg-Landau models). The method developed there, valid for so-called gradient systems is based on the superexponential estimate, which allows to replace empirical averages of local functions by functions of the density field. This implies that one can prove with the same method at the same time the hydrodynamic limit of the weakly asymmetric exclusion process as well as the large deviations from the hydrodynamic limit for the SSEP.

The study of systems with multiple conserved quantities and their hydrodynamic limits has gained substantial interest in recent times
see e.g. \cite{ABC-hd}, \cite{ABC-KPZ}, \cite{Frank-Casini-Cristian-Fluct,Hidde-Flucutations} and references therein (see also e.g. \cite{quastel1992diffusion} for an earlier reference). In particular, these results constitute rigorous versions of fluctuating hydrodynamics or mode coupling theory, \cite{spohn2015nonlinear}, \cite{spohn2014nonlinear}, see also
\cite{spohn2012large}.
Other motivation for multispecies (and also connected multi-layer) models and their scaling limits is the phenomenon of uphill diffusion \cite{casini2023uphill,floreani2022switching} and systems of active particles.

The process we study in our paper is the multispecies analogue of the SSEP, known as the multispecies stirring process \cite{zhou,vanicat2017exact,casiniDualityStirring}.
In this process, at every site there is at most one particle, which can be of type $\alpha, \alpha\in \{1,\ldots,n\}$. The absence of a particle is called a particle of type zero.
To each nearest neighbor edge is associated a Poisson clock of rate 1, different Poisson clocks being independent. When the clock of an edge rings, the occupancies of that edge are exchanged. An exchange between a particle of type $\alpha, \alpha\in \{1,\ldots,n\}$ at site $x$ and an empty site at
site $x+1$ is of course the same as a jump of the particle from $x$ to $x+1$.
It is well-known that the hydrodynamic limit for the densities of the $n$ types of particles is a system of uncoupled heat equations, and in \cite{Frank-Casini-Cristian-Fluct} it is also proved that the fluctuations around this hydrodynamic limit is an infinite dimensional Ornstein-Uhlenbeck process.
Other results on the multispecies stirring process include duality, and exact formulas for the moments in the non-equilibrium steady state of a boundary driven version using duality combined with integrability.
To our knowledge, no explicit formula exists for the large deviation rate function for the density profile in the non-equilibrium steady state, as is the case e.g. for the SEP, see \cite{derrida2002large}.
In the setting of the macroscopic fluctuation theory, the rate function in the non-equilibrium steady state is strongly related to the rate function for the trajectory of the empirical density profile, i.e., the large deviations around the hydrodynamic limit.
Therefore, in order to make progress in the understanding of non-equilibrium large deviations in multispecies models, it is natural to study the large deviations around the hydrodynamic limit for the multispecies stirring process.
To the best of our knowledge, no rigorous results have been established in the context of dynamic large deviations for the multispecies stirring process. In this paper we implement the method of \cite{KOV} for gradient systems,  based on the superexponential estimate,
(see also \cite{KipnisLandim} chapter 10) in our multispecies setting.
The study of the large deviation principle for the multispecies stirring process, relies on the introduction of a well-chosen weakly asymmetric process, where the rates are deformed by an exponential tilting, i.e., by introducing weak and slowly varying (in space) external fields that introduce a drift on the particles of various types. To understand the probability of  deviating trajectories for the densities, one has to choose these fields governing the asymmetry in such a way that in the modified dynamics, the deviating trajectory becomes typical. The large deviation rate function is then roughly the relative entropy of the modified dynamics w.r.t. original dynamics which can be computed with the Girsanov formula.  In particular, exactly as is done in \cite{KOV} for the single species case, we also prove as byproduct the hydrodynamic limit of this weakly asymmetric multispecies process, which is a system of nonlinear coupled parabolic equations, closely related to the ABC model \cite{ABC-hd}. In this limiting partial differential equation (PDE), in addition to diffusion, a drift term is introduced into the currents, which makes this system appealing for describing multi-component diffusion processes in applications \cite{rudan2015physics,colangeli2023uphill}.
The relation between the drift currents and the fields is via the symmetric Onsager matrix, which coincides with the covariance matrix of the multinomial reversible measures.

As a perspective towards further research, this work could serve as a starting point for various questions. These include to explore the extension of large deviation principles and hydrodynamic limits to boundary-driven systems in the multispecies setup, as previously done for the single species case \cite{LD-boundary}. Additionally one can investigate the density field fluctuations in the weakly asymmetric multispecies stirring process, analogous to what has been done for the ABC model in the context of the Kardar-Parisi-Zhang (KPZ) universality class \cite{ABC-KPZ}. Moreover, it would be of interest to apply these techniques to multi-layer systems \cite{Hidde-Flucutations}, where the geometry consists of two layers with SEP dynamics occurring within each layer, coupled with particle exchange between the layers.

\subsection{Organization of the paper and main results}
Starting from the literature, in Section \ref{model section} we first recall the definition of the multispecies stirring process, on the geometry of a torus, reporting also its reversible measure. Then, in Section \ref{wasep section} we define a weakly asymmetric version of the multispecies stirring process where the transition rates are perturbed by a family of potentials, indexed by the species involved in the transition and dependent on space and time. 

Finally, in Section \ref{SEE section} we state the so called \textit{superexponential estimate}. This estimate turns out to be a useful tool in the proof of the hydrodynamic limit of the weakly asymmetric model and in the proof of the large deviation principle as well. The proof of this estimate goes beyond the main scope of this paper, therefore we report it in Appendix \ref{proof of superexponential estimate}. 

In Section \ref{section 3} we state the hydrodynamic limit of the weakly asymmetric model. We postpone the proof to Appendix \ref{Appendix B} since it can be shown by standard methods. Then, in Section \ref{choice of fields}, we make a specific choice of potentials needed for the proof of the large deviation principle. This choice is further motivated by Einstein relations between diffusion, mobility and compressibility matrices.

In section \ref{LDP section}, we proceed to state and prove the large deviation principle. With both the original model and the weakly asymmetric model established, we first obtain the Radon-Nikodym derivative of their respective path-space measures in section \ref{radon-Nikodym section}. This can be computed using the Girsanov formula and will be equal to exponential martingale associated with the original model. For the upper bound, we first establish the exponential tightness of the path-space measures in section \ref{exponential tightness section} (which reduces the proof to verifying the upper bound for compact sets instead of closed sets). The upper bound is then derived in Section \ref{upper bound} using the martingale property of the Radon-Nikodym derivative.

For the lower bound, which we prove in Section \ref{lower bound}, we demonstrate the relationship between the large deviation rate function and the hydrodynamic limit of the weakly asymmetric model. Specifically, for every deviating path, we show the existence of a potential such that this path becomes typical under the weakly asymmetric dynamics. This leads to a new formulation of the large deviation rate functional, expressed as the norm of this potential in an appropriate Sobolev space. Finally, using this relationship, we are able to demonstrate the lower bound. 

\paragraph{Acknowledgments.}  The authors acknowledge financial support from PNRR MUR project  ECS\_ 0000033\_ECOSISTER, which also enabled a visit to TU Delft in April 2024. We are grateful to the reading group on large deviations in interacting particle systems, which included the authors, Jonathan Junné, Richard Kraaij, and Rik Versendaal, for valuable discussions. Many thanks go to Cristian Giardinà for insightful discussions and suggestions regarding Einstein relations. FC extends thanks to Cecilia Vernia and Luca Selmi for their helpful discussions. Research supported in part by GNFM-INdAM.
\section{The multispecies stirring process}\label{model section}
In this section, we describe the multispecies stirring  process. We first examine the symmetric case, then we define a weakly asymmetric version in which the transition rates are "deformed" through a potential.\\
In both cases, we consider the geometry of a one-dimensional torus with $N$ sites, denoted by $\mathbb{T}_{N} = \mathbb{Z}/N\mathbb{Z}$. Additionally, for simplicity, we consider the scenario with two types of particles (in addition to vacancies, also called holes). The occupation variable is denoted by $\bm{\eta}=(\eta_{\alpha}^{x})_{\alpha\in \{0,1,2\}\,, x\in \mathbb{T}_{N}}$, where $\eta_{\alpha}^{x}$ represents the number of particles or holes (labeled by $\alpha$) at site $x$. For any time $t\geq 0$, the configuration of the process is denoted by $\left(\bm{\eta}(t)\right)_{t\geq 0}$. \\
As a convention, we use the labels $\alpha=1,2$ to distinguish particles of species $1$ and $2$, and we use the label $\alpha=0$ to denote the holes. The term "holes" is motivated by the fact that its occupation variable is determined once we know the occupation variable of the species of particles $1$ and $2$, due to the so called "exclusion constraint"
\begin{align}\label{exclusion-constraint}
    \eta_{0}^{x}=1-\eta_{1}^{x}-\eta_{2}^{x}\qquad \forall x\in \mathbb{T}_{N}\,.
\end{align} 
Therefore, the configuration space reads
\begin{equation}\label{state space}
    \Omega=\bigotimes_{x=1}^{N}\Omega_{x}\qquad \text{where}\qquad \Omega_{x}=\left\{\eta^{x}=(\eta_{0}^{x},\eta_{1}^{x},\eta_{2}^{x})\;:\;\sum_{\alpha=0}^{2}\eta_{\alpha}^{x}=1\right\}\,.
\end{equation}
In the literature, the multispecies stirring process has also been considered with maximal occupancy per site higher than $1$ (see \cite{zhou,casiniDualityStirring}), and also the boundary driven case has been considered (see \cite{vanicat2017exact,casiniDualityStirring}).\\
In this paper, on the same geometry and configuration space, we introduce two types of dynamics: the symmetric and the weakly asymmetric ones. In the symmetric dynamics, each transition occurs at the same rate to both the left and the right. In contrast, the weakly asymmetric dynamics introduces a weak asymmetry in the rates, resulting in a "drift" in the particles' jumps.
\subsection{The symmetric case}
In the symmetric case, the dynamics consists in swapping occupancies of nearest neighbor sites according to independent rate 1 Poisson processes. More precisely, considering any bond $(x,x+1)$, any particle or hole present at site $x$ is exchanged with any particle or hole present at site $x+1$. For any $\alpha,\beta \in \{0,1,2\}$ such that $
\eta^x_\alpha \eta^{x+1}_\beta = 1$ we now define the configuration $\bm \eta^{x,x+1}_{\alpha,\beta}$ obtained by swapping the occupancies at $x$ and $x+1$ via
\[
\bm\eta^{x,x+1}_{\alpha,\beta}(y) = 
    \bm\eta^x_\alpha - \delta_\alpha^x + \delta^x_\beta - \delta^{x+1}_\beta +\delta^{x+1}_\alpha 
\]
where $\pm \delta_{\alpha}^{x}$ indicates that a particle or vacancy denoted by $\alpha$ is added or removed at site $x$. If $
\eta^x_\alpha \eta^{x+1}_\beta \neq 1$ then we make the convention that $\bm\eta^{x,x+1}_{\alpha,\beta} = \bm \eta$. The infinitesimal generator of this process is then given by a superposition of local operators as
\begin{equation}\label{symmetirc-genrator}
    \mathcal{L}=\sum_{x=1}^{N}\mathcal{L}_{x,x+1}\qquad \text{where}\qquad \mathcal{L}_{N,N+1}=\mathcal{L}_{N,1}
\end{equation}
where, for every local function $f:\Omega\to \mathbb{R}$
\begin{equation}
    \mathcal{L}_{x,x+1}f(\bm{\eta})=\sum_{\alpha,\beta=0}^{2}\eta_{\alpha}^{x}\eta_{\beta}^{x+1}\left(f(\bm{\eta}^{x,x+1}_{\alpha,\beta})-f(\bm{\eta})\right)\,.
\end{equation}
This process has  reversible product measures with   multinomial marginals:
\begin{equation}\label{reversible-MS-symm}
    \nu^{\bm p}_{N}=\bigotimes_{x=1}^{N}\nu_{N,x}^{\bm p}\qquad \nu_{N,x}^{\bm p}\sim \text{Multinomial}(1,p_{1},p_{2})\,.
\end{equation}
Here $\bm p = (p_1,p_2)$ where $p_1$ and $p_2$ are the probabilities of having a particle of type $1$ respectively $2$ at any given site. Furthermore, under the measure $\nu_N^{\bm p}$, the probability to have no particle at any site $x \in \mathbb{T}_N$ is equal to $p_0 := 1-p_1-p_2$.
Reversibility of the measures $\nu_N^{\bm p}$ follows from the  detailed balance condition. In the following, it will be useful to denote by $\nu_{N}^{1/3}$ the reversible measure with multinomial densities given by $p_{1} = p_{2} = \frac{1}{3}$. 

We denote by $\mathbb{T}=[0,1]$ the one-dimensional torus. For $\bm{\gamma}(\cdot) = (\gamma_{1}(\cdot), \gamma_{2}(\cdot))$ where  $\gamma_{\alpha}: \mathbb{T} \to [0,1]$ for $\alpha \in \{1,2\}$ are smooth functions,  we introduce the local equilibrium product measures associated to $\bm \gamma$: 
\begin{align}\label{reversible-measure-withProfile}
    \nu^{\bm{\gamma}(\cdot)}_{N}:=\bigotimes_{x=1}^{N}\nu^{\bm{\gamma}(\cdot)}_{N,x},
\end{align}
where the marginals over each site are multinomial and given by
\begin{align}
    \nu^{\bm{\gamma(\cdot)}}_{N,x}\left(\left\{\bm{\eta}\,:\, \eta_{\alpha}^{x}=1\right\}\right)=\begin{cases}
        \gamma_{\alpha}\left(\frac{x}{N}\right)\qquad &\text{if}\quad \alpha\in \{1,2\},\\
        1-\gamma_{1}\left(\frac{x}{N}\right)-\gamma_{2}\left(\frac{x}{N}\right)&\text{if}\quad \alpha=0.
        \end{cases}
\end{align}
\subsection{The weakly asymmetric stirring process}\label{wasep section}
We introduce a weakly asymmetric version of the multispecies stirring process, which will play a crucial role in the study of large deviations. We parametrize the weak asymmetry by three smooth functions $\bm{H}=(H_{01},H_{02},H_{12})$. Moreover, we define for $\alpha < \beta$
\begin{equation}\label{requirement-fields}
    H_{\beta \alpha}(u,t) := -H_{\alpha\beta}(u,t)\qquad \forall u\in \mathbb{T},\quad \forall t\in [0,T]\,.
    \end{equation}
    The reason for this antisymmetric choice \eqref{requirement-fields} will be clarified later.  The time-dependent generator of the weakly asymmetric multispecies stirring process parametrized by $\bm H$ is then given by
    \begin{equation}\label{weak-asymmetric-genrator}
        \mathcal{L}^{\bm{H}}(t)=\sum_{x=1}^{N}\mathcal{L}_{x,x+1}^{\bm{H}}(t)\qquad \text{such that}\qquad \mathcal{L}_{N,N+1}^{\bm{H}}(t)=\mathcal{L}_{N,1}^{\bm{H}}(t)
    \end{equation}
    where the local operators are defined as
\begin{equation}\label{weak-asymmetric-genrator-densities}
\begin{split}
    \mathcal{L}^{\bm H}_{x,x+1}(t)f(\bm{\eta})=\sum_{\alpha,\beta=0}^{2}&c^{\bfH,\alpha\beta}_{(x,x+1)}(t)\left(f(\bm{\eta}^{x,x+1}_{\alpha,\beta})-f(\bm{\eta})\right)
    \end{split}\,,
\end{equation}
and where
\begin{equation}\label{WA-transitionRates}
    c^{\bfH,\alpha\beta}_{(x,x+1)}(t)
    =\exp\Big(\nabla_NH_{\alpha\beta}(\tfrac{x}{N},t)\Big)\eta_{\alpha}^{x}\eta_{\beta}^{x+1}.
\end{equation}
Here $\nabla_N$ denotes the discrete gradient, i.e., 
\begin{equation}
    \nabla_NH_{\alpha\beta}(\tfrac{x}{N},t) = H_{\alpha\beta}(\tfrac{x+1}{N},t) - H_{\alpha\beta}(\tfrac{x}{N},t).
\end{equation}
Later on we will omit the explicit dependence on $t$ in \eqref{weak-asymmetric-genrator-densities} for notational simplicity.
\begin{remark}
    In this remark we explain the choice imposed by \eqref{requirement-fields}. 
    Exchanging occupancy of type $\alpha$ at site $x$ with  type $\beta$ at site $x+1$ is identical with exchanging occupancy of type $\beta$ at site $x+1$ with  type $\alpha$ at site $x$. Therefore we must impose the following condition on the rates
    \begin{equation}\label{requirement-H-onRates}
        c^{\bfH,\alpha\beta}_{(x,x+1)}(t)=c^{\bfH,\beta,\alpha}_{(x+1,x)}(t).
    \end{equation}
For the rates defined in   \eqref{WA-transitionRates}, \eqref{requirement-H-onRates} is satisfied if  $H_{\alpha\beta}=-H_{\beta\alpha}$.
\end{remark}
We introduce some further notation. For all $T>0$, we consider the Skorokhod space {\color{black}$D\left([0,T],\Omega\right)$, which consists of the càdlàg trajectories taking values in $\Omega$.} On this space, we define the following path space measures:
\begin{itemize}
    \item $\mathbb{P}_{N}^{1/3}$: path space measure of the symmetric process with generator \eqref{symmetirc-genrator}, initialized with the distribution $\nu_{N}^{1/3}$.
    \item  $\mathbb{P}_{N}^{\bm{\gamma}}$: path space measure of the symmetric process with generator \eqref{symmetirc-genrator}, initialized with the distribution $\nu^{\bm{\gamma}(\cdot)}_{N}$.
    \item $\mathbb{P}_{N}^{\bm{\gamma},\bm{H}}$: path space measure of the weakly asymmetric process with generator \eqref{weak-asymmetric-genrator}, initialized with the distribution $\nu^{\bm{\gamma}(\cdot)}_{N}$.
\end{itemize}
For each species $\alpha\in \{1,2\}$ we introduce the corresponding empirical density field 
\begin{equation}\label{empirical-measure}
    \mu_{\alpha,N}(\bm{\eta}(N^{2}s)):=\frac{1}{N}\sum_{x=1}^{N} \eta_{\alpha}^{x}({\color{black}N^{2}s})\delta_{\frac{x}{N}}\,,
\end{equation}
\begin{remark}
    For the sake of notational simplicity, sometimes we will abbreviate $\mu_{\alpha,N}(\bm{\eta}(N^{2}s))$ by ${\mu}_{\alpha,N}(s)$. 
\end{remark}

This density field
 $\mu_{\alpha,N}(s)$ takes values in $D\left([0,T],\mathcal{M}_{1}\right)$, where $\mathcal{M}_{1}$ denotes the space of measures over $\Omega$ with total mass bounded by 1, i.e., $\sup_{||f||\leq 1} \langle \mu_{\alpha,N}(s),f\rangle \leq 1$. Additionally, we define the vector of density fields
\begin{equation}
    \bm{\mu}_{N}(s)=\begin{pmatrix}
        \mu_{1,N}(s)\\
        \mu_{2,N}(s)
    \end{pmatrix}
\end{equation}
taking values in the space $D\left([0,T],\mathcal{M}_{1}\times \mathcal{M}_{1}\right)$. We consider two functions $G_{1},G_{2}\in C^{2,1}(\mathbb{T}\times [0,T])$ and we list them in a vector denoted by 
\begin{equation}
    \bm{G}(u,s):= \begin{pmatrix}
        G_{1}(u,s)\\
        G_{2}(u,s)
    \end{pmatrix}\,.
\end{equation}
Then, we denote the pairing   
\begin{align}
    \langle\bm{\mu}_{N}(s),\bm{G}(\cdot,s)\rangle=\int_{\mathbb{T}}G_{1}(u,s) \mu_{1,N}(\dd u, s)+\int_{\mathbb{T}}G_{2}(u,s)\mu_{2,N}(\dd u,s)\,.
\end{align}
\subsection{Superexponential estimate}\label{SEE section}
In this section we state the so called \textit{superexponential estimate}. This is a crucial tool initially introduced in \cite{guo1988nonlinear}, 
\cite{KOV}, which allows to replace macroscopic averages of local observables by an appropriate function of the local density. This is crucial both in the derivation of the hydrodynamic limit of the weakly asymmetric model as well as in the large deviations of the symmetric model. In the latter it becomes important that the replacement is superexponentially good, i.e., can still be performed e.g.\ in exponential martingales containing local averages.  This replacement is carried out within a space interval constructed around a microscopic point. Eventually, the size of this interval shrinks as the system size increases. 

We consider a function $\phi \in C(\Omega)$ and we define 
\begin{align}\label{phi-tilda}
    \widetilde{\phi}(\bm p ):=\mathbb{E}_{\nu_{N}^{\bm p}}[\phi]\,.
\end{align}
namely the expectation with respect to the product over sites of multinomial distribution $\nu_{N}^{\bm p}$ with constant parameters $\bm{p}=(p_{1},p_{2})$.\\
Next, we introduce a function that will play a key role in the superexponential estimate. This function relates to the behavior of occupation variables in a small neighborhood around a microscopic point and it reads
\begin{align}\label{definition-V}
    V_{N,\epsilon}(\bm{\eta})=\sum_{x=1}^{N}\left|\frac{1}{2\epsilon N+1}\sum_{|x-y|\leq N\epsilon}\tau_{y}\phi(\bm{\eta})-\widetilde{\phi}\left(\frac{1}{2N\epsilon+1}\sum_{|x-y|\leq \epsilon N}\eta_{1}^{y},\frac{1}{2N\epsilon+1}\sum_{|x-y|\leq \epsilon N}\eta_{2}^{y}\right)\right|\,.
\end{align}
The superexponential estimate is then the following result. 
\begin{theorem}\label{superexponential estimate}
For any $\delta>0$, for all $T>0$ and $\phi \in C(\Omega)$
\begin{align}\label{supo}
    \uplim_{\epsilon\to \infty}\;\uplim_{N\to\infty}\frac{1}{N}\log\;\mathbb{P}_{N}^{1/3}\left(\frac{1}{N}\int_{0}^{T}V_{N,\epsilon}(\bm{\eta}(s))\dd s\geq \delta\right)=-\infty.
\end{align}
\end{theorem}
\noindent Since the proof of this Theorem is rather long and involved, and it is not the main result of this paper, we postpone it to appendix \ref{proof of superexponential estimate}.

In the following corollary we show that the superexponential estimate also holds when we start from a local equilibrium distribution.
\begin{corollary}
Given a profile $\bm{\gamma}=(\gamma_{1},\gamma_{2})$, \eqref{supo} holds also for the path space measure $\mathbb{P}_{N}^{\bm\gamma}$.   
\end{corollary}
\textbf{Proof}:
the proof follows from Theorem \ref{superexponential estimate} and from the following upper bound for all sets $A\subset \Omega_{N}$
\begin{align}
    \mathbb{P}_{N}^{\bm\gamma}(A)=\sum_{\bm{\eta}\in \Omega_{N}}\frac{d \nu_{\bm\gamma}^{N}}{d \nu_{N}^{1/3}}(\bm{\eta})\mathbb{P}_{N}^{\bm{\eta}}(A)\nu_{N}^{1/3}(\bm{\eta})\leq 3^{N}\mathbb{P}_{N}^{1/3}(A).
\end{align}
\begin{flushright}
    $\square$
\end{flushright}

\section{Hydrodynamic limit of the weakly asymmetric model}\label{section 3}
In this section we state the hydrodynamic limit of the weakly asymmetric version of the multispecies stirring model with generator \eqref{weak-asymmetric-genrator}. 
\begin{remark}\label{notation-convention-rho}
    Sometimes, in this section and in the following one, in order to alleviate the notation, we do not explicitly write the space and time dependence of the densities. Namely, when this dependence is understood we only write $\bm{\rho},\rho_{1},\rho_{2}$ in place of $\bm{\rho}(u,t),\rho_{1}(u,t),\rho_{2}(u,t)$.  The same convention is used for the potentials $H_{\alpha,\beta}(u,t)$. 
\end{remark}
\begin{theorem}\label{hydro wasep}
   As $N$ tends to infinity, the density fields for the species $\alpha=1,2$ converge in probability $\mathbb{P}_{N}^{\gamma,H}$ to the unique weak solution $(\rho_{1}(t,u),\rho_{2}(t,u))$ of the following system of hydrodynamic equations
\begin{align}\label{HD-equations}
    \partial_{t}\rho_{1}&=\Delta \rho_{1}-2\nabla\left(\rho_{1}(1-\rho_{1}-\rho_2)\nabla H_{10}\right)-\nabla\left(2\rho_{1}\rho_{2}\nabla H_{12}\right),\nonumber\\
    \partial_{t}\rho_{2}&=\Delta \rho_{2}-2\nabla\left(\rho_{2}(1-
    \rho_1 - \rho_{2})\nabla H_{20}\right)+\nabla\left(2\rho_{1}\rho_{2}\nabla H_{12}\right),
\end{align}
with initial conditions 
\begin{align}
    \rho_{1}(0,u)=\gamma_{1}(u),\qquad    \rho_{2}(0,u)=\gamma_{2}(u).
\end{align}
\end{theorem}
\begin{proof}
    For the proof we refer to Section \ref{Appendix B}. 
\end{proof}
In particular in the case where every $H_{\alpha\beta} = 0$, we recover the uncoupled heat equations
\begin{align}\label{HD-equations uncoupled}
    \partial_{t}\rho_{1}&=\Delta \rho_{1},\\
    \partial_{t}\rho_{2}&=\Delta \rho_{2}.
\end{align}
which in matrix form reads 
\begin{equation}
    \partial_t \left(\begin{matrix} 
    \rho_1\\\rho_2
    \end{matrix}\right) = D(\rho_1,\rho_2) 
    \left(\begin{matrix}
        \Delta \rho_1\\ \Delta \rho_2
    \end{matrix}\right),
\end{equation}
where 
\begin{equation}\label{diffusivity-matrix}
 D(\rho_{1},\rho_{2})=   \begin{pmatrix}
        1&0\\
        0&1
    \end{pmatrix}\,
\end{equation}
is the diffusion matrix.

\subsection{Potentials for large deviations}\label{choice of fields}
    In order to prove the large deviations for the trajectory of the empirical densities, we need appropriate perturbations of the dynamics which make these deviating trajectories typical.
    As will become clear in section 4, these perturbations correspond to the weakly asymmetric stirring process, with 
    {\color{black}potentials} which we denote by 
    \begin{equation}\label{condition-field-1}
        H_{1}(u,t):=H_{10}(u,t),\qquad H_{2}(u,t):=H_{20}(u,t)\qquad \forall u\in \mathbb{T}\quad \text{and}\quad t\in [0,T]\,,
    \end{equation}
    and where moreover, the {\color{black}potential} $H_{12}$ must satisfy
     \begin{align}\label{condition-Field-2}
         H_{12}(u,t)=H_{1}(u,t)-H_{2}(u,t)\qquad \forall u\in \mathbb{T}\quad \text{and}\quad t\in [0,T].
     \end{align}
     Therefore, the resulting hydrodynamic equations read
    \begin{align}\label{HD-equations LDP}
    \partial_{t}\rho_{1}=&\Delta\rho_{1}-2\nabla \left(\rho_{1}(1-\rho_{1})\nabla H_1\right)+2\nabla \left(\rho_{1}\rho_{2}\nabla H_2\right),\nonumber \\
    \partial_{t}\rho_{2}=&\Delta\rho_{2}-2\nabla \left(\rho_{2}(1-\rho_{2})\nabla H_2\right)+2\nabla \left(\rho_{1}\rho_{2}\nabla H_1\right)\,.
    \end{align}
The intuitive interpretation of this choice of {\color{black}potentials} is the following. Particles of type 1 and 2 are driven across the holes (particles of type $0$) by the {\color{black}force depending on the potentials} $H_1$ and $H_2$ {\color{black}(namely the external fields are given by the gradient of the potentials)} respectively. When two particles of type 1 and 2 are adjacent, a competition between the fields generated by the potentials $H_{1}$ and $H_{2}$ sets in. As a result, the net field acting on each species is given by $\pm \nabla(H_{1}-H_{2})$ respectively. Moreover, as we will point out later, this choice of the fields allows the system to satisfy the Einstein relation connecting diffusion, mobility and compressibility matrices. 

\subsection{Currents and the Einstein relation}
\paragraph{Macroscopic currents.}The hydrodynamic equations \eqref{HD-equations LDP} can be interpreted as conservation laws. To illustrate this, we compute the macroscopic currents for each species. These currents represent the net flux crossing an infinitesimal volume surrounding a point $u \in \mathbb{T}$ at any time $t \in [0, T]$. We identify two types of currents:
\begin{enumerate}
    \item \textit{Fick's currents}: These currents are proportional to minus the density gradients via the diffusion matrix as given in \eqref{diffusivity-matrix}.
The currents are expressed as
\begin{align}
        \begin{pmatrix}
        J_{1}^{F}\\
        J_{2}^{F}
    \end{pmatrix}=-D(\rho_{1},\rho_{2})\begin{pmatrix}
        \nabla\rho_{1}\\
        \nabla\rho_{2}
    \end{pmatrix}\,.
\end{align}
Generally, the diffusivity matrix \eqref{diffusivity-matrix} may depend on the densities, but in this case, it simplifies to the identity matrix.
\item \textit{Drift currents}: these currents are defined as the product of (twice)\footnote{The factor $2$ in front is due to the fact that in the generator \eqref{symmetirc-genrator} both jumps, to the left and to the right, have rate $1$, instead of $1/2$.} the mobility matrix 
\begin{equation}\label{mobility-matrix}
    \chi(\rho_{1},\rho_{2})=\begin{pmatrix}
        \rho_{1}(1-\rho_{1})&-\rho_{1}\rho_{2}\\
        -\rho_{1}\rho_{2}&\rho_{2}(1-\rho_{2})
    \end{pmatrix}
\end{equation}
and the external field, which is the gradient of the {\color{black}potential} $(H_{1}, H_{2})$. Specifically, these currents are given by
\begin{equation}
    \begin{pmatrix}
        J_{1}^{D}\\
        J_{2}^{D}
    \end{pmatrix}=2\chi(\rho_{1},\rho_{2})\begin{pmatrix}
        \nabla H_{1}\\
        \nabla H_{2}
    \end{pmatrix}\,.
\end{equation}
It is important to note that the mobility matrix \eqref{mobility-matrix} is symmetric and corresponds to the covariance matrix of the multinomial distribution with parameters $\rho_1, \rho_2$. This matrix also appears in the study of fluctuations as proved in \cite{Frank-Casini-Cristian-Fluct}. 
\end{enumerate}


We now compute the total currents, which are given by the sum of Fick's and of the Drift currents for each species. Namely they read 
\begin{equation}\label{totalCurrent}
     \begin{pmatrix}
        J_{1}\\
        J_{2}
    \end{pmatrix}= \begin{pmatrix}
        J_{1}^{F}\\
        J_{2}^{F}
    \end{pmatrix}+ \begin{pmatrix}
        J_{1}^{D}\\
        J_{2}^{D}
    \end{pmatrix}\,.
\end{equation}
Therefore, equation \eqref{HD-equations LDP} can be obtained by substituting the total currents \eqref{totalCurrent} in the continuity equations of the densities, i.e. 
\begin{align}
    \partial_{t}\rho^{1}&=-\nabla J_{1},\nonumber\\
        \partial_{t}\rho^{2}&=-\nabla J_{2} \,.
\end{align}
\paragraph{Einstein's relation.} 
 We introduce the free energy functional, that is defined as the large deviation functional of a multinomial random variable with number of trials equal to $1$ and probabilities all equal to $1/3$. Namely, we have that 
\begin{equation}
    F(\rho_{1},\rho_{2})=\rho_{1}\log(\rho_{1})+\rho_{2}\log(\rho_{2})+(1-\rho_{1}-\rho_{2})\log(1-\rho_{1}-\rho_{2})+\log(3)\,.
\end{equation}
We compute the Hessian matrix of $ F(\rho_{1},\rho_{2})$, sometimes called the inverse of the compressibility matrix, obtaining
\begin{equation}\label{Hessian}
   F^{''}(\rho_{1},\rho_{2})= \begin{pmatrix}
        \frac{1}{\rho_{1}}+\frac{1}{1-\rho_{1}-\rho_{2}}&\frac{1}{1-\rho_{1}-\rho_{2}}\\
        \frac{1}{1-\rho_{1}-\rho_{2}}&\frac{1}{\rho_{2}}+\frac{1}{1-\rho_{1}-\rho_{2}}
    \end{pmatrix}\,.
\end{equation}
Then we see that by combining \eqref{diffusivity-matrix}, \eqref{mobility-matrix} and \eqref{Hessian}, the following relation holds. 
\begin{equation}\label{Einstein-relation}
    D(\rho_{1},\rho_{2})=F^{''}(\rho_{1},\rho_{2})\chi(\rho_{1},\rho_{2})\,.
\end{equation}
This equality is called the Einstein relation (see \cite{jonaLansinio,spohn2012large} for details).
Notice that we used the specific form of the potentials described in \eqref{condition-field-1} and \eqref{condition-Field-2} to obtain the Einstein relation  \eqref{Einstein-relation}, which provides another physical motivation for these conditions.
\begin{remark}\label{recover remark}
We can recover the hydrodynamic limit of the single species weakly asymmetric exclusion process from equation \eqref{HD-equations LDP} as given in \cite[Theorem 3.1]{KOV}. Namely, if we choose the same {\color{black}potential} $H_1=H_2=H$, then we obtain the following.
\begin{align}
    \partial_{t}\rho_{1}=&\Delta\rho_{1}-2\nabla \left(\rho_{1}(1-\rho_{1}-\rho_{2})\nabla H\right),\nonumber\\
    \partial_{t}\rho_{2}=&\Delta \rho_{2}-2\nabla \left(\rho_{2}(1-\rho_{1}-\rho_{2})\nabla H\right).
\end{align}
By now defining $\varrho := \rho_1+\rho_2$, i.e., $\varrho$ does not distinguish between particles of type 1 and type 2, then $\varrho$ satisfies 
\begin{equation}
\partial_t \varrho = \Delta \varrho - 2\nabla(\varrho(1-\varrho)\nabla H).
\end{equation}
This result is to be expected, since the process defined as $\eta := \eta_1 + \eta_2$ is a standard (weakly asymmetric) exclusion process.
\end{remark}
\begin{remark}\label{generalization remark HD}
    At the cost of more notational complexity, but no additional mathematical difficulties, one can generalize the hydrodynamic limit of Theorem \ref{hydro wasep} to any number of species, i.e., $\alpha \in \{0,1,...,n\}$ for any $n\in\mathbb{N}$.

The hydrodynamic limit of the weakly asymmetric model with the general {\color{black}potentials} $H_{\alpha \beta} = -H_{\beta \alpha}$  is now given by a system of $n$ dependent partial differential equations 
\begin{equation}
\partial_t \rho_\alpha = \Delta \rho_\alpha -2\sum_{\beta \neq \alpha} \nabla \left(\rho_\alpha \rho_\beta \nabla H_{\alpha\beta} \right),
\end{equation}
with the convention that $\rho_0 = 1-\sum_{\alpha=1}^n \rho_\alpha$. 
For the large deviations of the trajectories of the densities  we only need $n$ {\color{black}potentials}. The choice of potentials, which is the analogue of the conditions \eqref{condition-field-1} and \eqref{condition-Field-2}, then reads 
\begin{equation}
H_{\alpha} := H_{0\alpha}, \ \ \ \ \ \ H_{\alpha \beta} := H_\alpha - H_\beta. 
\end{equation}
This choice of potentials then results in the following hydrodynamic limit 
\begin{equation}
\partial_t \rho_\alpha = \Delta \rho_\alpha -2\nabla \left(\rho_\alpha (1-\rho_\alpha) \nabla H_\alpha\right) - 2 \sum_{\beta \neq \alpha} \nabla \left(\rho_\alpha \rho_\beta \nabla H_\beta   \right).
\end{equation}
\end{remark}

\section{Large deviations}\label{LDP section}
In this section we aim to prove the large deviation principle of the multispecies stirring process. We start by defining the rate function $\calI_{\bm\gamma} : D([0,T], \mathcal{M}_1\times \mathcal{M}_1) \to [0,\infty]$ which consists of two parts 
\begin{equation}
\calI_{\bm\gamma}(\bfrho) = h(\bfrho(0);\bm\gamma) + \calI_0(\bfrho),
\end{equation}
{\color{black}where $\bm{\rho}(0)$ denotes the trajectory $\bm{\rho}$ evaluated at the initial time $t=0$.} Here $h(\bfrho(0);\bm\gamma)$ is the static part of the large deviation functional, i.e. the one due to the initial product measure $\nu_N^{\bm\gamma}$ as defined in \eqref{reversible-measure-withProfile}. It is given by the formula
\begin{align}\label{static}
h(\bfrho(0);\bm\gamma) = \sup_{\bm\phi}h_{\bm\phi}(\bfrho(0);\bm\gamma), \qquad \qquad h_{\bm\phi}(\bfrho(0);\bm\gamma)  =  \sum_{\alpha=0}^2 \langle \rho_{\alpha}(0),\phi_\alpha\rangle - \int_{\mathbb{T}}\log\left(\sum_{\alpha=0}^2 \gamma_\alpha(u) e^{\phi_\alpha(u)}\right)\dd u,
\end{align}
where the supremum is taken over all continuous $\bm\phi = (\phi_0,\phi_1,\phi_2)$ and we use  that $\rho_0 := 1-\rho_1-\rho_2$.

$\calI_0(\bfrho)$ is the dynamic part of the large deviation functional, i.e., the one due to the dynamics of the trajectory $\bfrho$ over time. It has the following form,
\begin{align}\label{I-zero}
    \mathcal{I}_0(\bfrho) = \sup_{\bfG} \left\{\ell(\bfrho;\bfG) - \tfrac{1}{2}||\bfG||_{\mathcal{H}(\bfrho)}^2\right\}.
\end{align}
Here the supremum is taken over vectors of functions $\bfG = \binom{G_1}{G_2}$ where both $G_1,G_2 \in C^{2,1}(\mathbb{T}\times [0,T])$. The operator $\ell$ is the linear operator corresponding to the hydrodynamic limit of the multispecies SEP, i.e., it is given by 
\begin{equation}
\ell(\bfrho;\bfG) = \langle \bfrho(T),\bfG(\cdot, T)\rangle - \langle \bfrho(0),\bfG(\cdot, 0)\rangle - \int_0^T \langle \bfrho(t), (\partial_t + \Delta)\bfG(\cdot, t)\rangle\dd t,
\end{equation}
which is equal to zero for all $\bfG$ iff $\bfrho$ solves the PDE $\dot{\bfrho}(t) = \Delta \bfrho(t)$ in the sense of distributions. Lastly, the norm in the definition of the rate function \eqref{I-zero} is the norm of the Hilbert space $\mathcal{H}(\bfrho)$ equipped with the following inner product 
\begin{align}\label{scalar-prduct-H}
\langle \bfG, \bfH\rangle_{\mathcal{H}(\bfrho)}  \nonumber
&= 2\int_0^T \left<\rho_1(t)(1-\rho_1(t)), \nabla G_1(\cdot,t) \nabla H_1(\cdot,t)\right> \dd t\\
&\ \ \ \ +2\int_0^T \left<\rho_2(t)(1-\rho_2(t)), \nabla G_2(\cdot,t) \nabla H_2(\cdot,t)\right> \dd t\nonumber\\
&\ \ \ \ -2\int_0^T \left<\rho_1(t)\rho_2(t), \nabla G_1(\cdot,t) \nabla H_2(\cdot,t) + \nabla G_2(\cdot,t) \nabla H_1(\cdot,t)\right> \dd t.
\end{align}
\begin{remark}
    In Lemmas \ref{explicit static} and \ref{lemma51} we give more explicit forms of the functionals $h(\cdot;\bm\gamma)$ and  $\calI_0$ respectively. Namely, $h(\bfrho(0);\bm\gamma)$ can be written as the limit of relative entropies of multinomials with respective densities $\bfrho(0)$ and $\bm\gamma$, and $\calI_0(\bfrho) = \frac{1}{2}||\bfH||_{\mathcal{H}(\bfrho)}^2$ where $\bfH\in\mathcal{H}(\bfrho)$ is the unique function such that $\bfrho$ satisfies \eqref{HD-equations LDP} in the weak sense.  
\end{remark}

In order for a large deviation principle to hold, we need to show that we have the following two inequalities:
\begin{itemize}
    \item \textbf{Upper bound}: For every closed $\mathcal{C} \subset D([0,T];\mathcal{M}_1\times \mathcal{M}_1)$ we have that 
    \begin{equation}
    \uplim_{N\to\infty} \frac{1}{N} \log \p^{\bm\gamma}_N\left( \bm\mu_N \in \mathcal{C}\right) \leq -\inf_{\bm\rho \in \mathcal{C}} \calI_{\bm\gamma}(\bm\rho)
    \end{equation}
    \item \textbf{Lower bound}: For every open $\mathcal{O}\subset D([0,T];\mathcal{M}_1\times \mathcal{M}_1)$ we have that 
    \begin{equation}
    \lowlim_{N\to\infty} \frac{1}{N} \log \p^{\bm\gamma}_N\left( \bm\mu_N \in \mathcal{O}\right) \geq -\inf_{\bm\rho \in \mathcal{O}} \calI_{\bm\gamma}(\bm\rho)
    \end{equation}
\end{itemize}
We give a proof for the upper and lower bound in sections \ref{upper bound} and \ref{lower bound} respectively. 
First  we calculate the Radon-Nikodym derivative $\frac{\dd\p^{\rho,H}_N}{\dd\p^{\bm\gamma}_N}$ of the path space measures of the weakly asymmetric process relative to the original process in section \ref{radon-Nikodym section}. Additionally, we establish   exponential tightness in section  \ref{exponential tightness section} which allows for the substitution of closed sets with compact sets in the derivation of the upper bound.
\begin{remark}
    Often, in the following to alleviate notation we do not write explicitly the time dependence of the occupation variables. Namely, we write  $\eta_{\alpha}^{x}$ in place of $\eta_{\alpha}^{x}(N^{2}s)$ when the time dependence is understood.
\end{remark}

\subsection{Radon-Nikodym derivative and the exponential martingale}\label{radon-Nikodym section}
The goal of this section is to obtain an explicit expression of the Radon-Nikodym derivative of the path space measure $\mathbb{P}_{N}^{\bm{H},\gamma}$ with respect to the path space measure $\mathbb{P}_{N}^{\gamma}$. From the literature (see \cite{palmowski,KipnisLandim}) the Girsanov formula states that 
\begin{align}\label{girsanov-general}
    \log\left(\frac{\dd P_{N}^{\bm H,\gamma}}{\dd P_{N}^{\bm\gamma}}\right)=&\sum_{x=1}^{N}\sum_{\alpha,\beta=0}^{2}\int_{0}^{T}\log\left(\frac{c^{\bfH,\alpha\beta}_{(x,x+1)}(s)}{\eta_{\alpha}^{x}\eta_{\beta}^{x+1}}\right)\dd J_{\alpha\beta}^{x,x+1}(s)\nonumber
    \\-&N^{2}\sum_{x=1}^{N}\sum_{\alpha,\beta=0}^{2}\int_{0}^{T}\eta_{\alpha}^{x}\eta_{\beta}^{x+1}\left(\exp{\left\{\nabla_N H_{\alpha\beta}\left(\tfrac{x}{N},s\right)\right\}}-1\right)\dd s\,.
\end{align}
Here, we represent by $J_{\alpha,\beta}^{x,x+1}(s)$ the number of transitions occurred up to time $s\in[0,T]$ that swap the occupancies of species $\alpha,\beta$ between sites $x$ and $x+1$. Under the path space measure $\mathbb{P}_{N}^{\bm{H},\gamma}$ the random process $\left(J_{\alpha,\beta}^{x,x+1}(s)\right)_{s\geq 0}$ is a Poisson process with intensity $c^{\bfH,\alpha\beta}_{(x,x+1)}(s)$.
In the following result we provide an alternative formula for the Radon-Nikodym derivative defined in \eqref{girsanov-general}. 
\begin{lemma} For all $T\geq 0$, for all $N\in \mathbb{N}$ and for all $H_{1},H_{2}\in C^{2,1}(\mathbb{T}\times [0,T])$ we have that 
\begin{align}\label{exponetial-martingale}
Z_{T,N}^{\bm H}(\bm\mu_N) :=\frac{dP_{N}^{\bm H,\gamma}}{dP_{N}^{\bm\gamma}}= &\exp\left(N\langle \bm\mu_N(T), \bm H(\cdot,T)\rangle  - N\langle \bm\mu_N(0), \bm H(\cdot,0)\rangle\right)\nonumber
\\&\cdot
\exp\left( - \int_0^T e^{-N\langle\bm\mu_N(s), \bm H(\cdot,s)\rangle}\left(\partial_{s}+ N^2\mathcal{L}\right) e^{N\langle \bm\mu_N(s), \bm H(\cdot,s)\rangle}\dd s\right)
\end{align}
Additionally,  under conditions \eqref{requirement-fields}, \eqref{condition-field-1} and \eqref{condition-Field-2} we have
\begin{align}\label{action-generator-exponential}
    &N^{2}e^{-N\langle\bm\mu_N(s), \bm H(\cdot,s)\rangle}\mathcal{L}e^{N\langle\bm\mu_N(s), \bm H(\cdot,s)\rangle}
    =
    \sum_{x=1}^{N}\sum_{\alpha,\beta=0}^{2}\eta_{\alpha}^{x}\eta_{\beta}^{x+1}\left[\exp{\left\{\nabla_N H_{\alpha\beta}\left(\tfrac{x}{N},s\right)\right\}}-1\right]\,,
    \end{align} 
and where
\begin{align}e^{-N\langle\bm\mu_N(s), \bm H\rangle}\partial_{s} e^{N\langle \bm\mu_N(s), \bm H(\cdot,s)\rangle}=\langle \mu_{1}(s),\partial_{s}H_{1}(\cdot,s)\rangle+\langle \mu_{2}(s),\partial_{s}H_{2}(\cdot,s)\rangle\,.
\end{align}
\end{lemma}

\begin{proof}
 We consider the first term in the right hand side of equation \eqref{girsanov-general} {\color{black}and we write 
\begin{align}
\sum_{x=1}^{N}\sum_{\alpha,\beta=0}^{2}\int_{0}^{T}\log\left(\frac{c^{\bfH,\alpha\beta}_{(x,x+1)}(s)}{\eta_{\alpha}^{x}\eta_{\beta}^{x+1}}\right)\dd J_{\alpha\beta}^{x,x+1}(s)\nonumber
=&
\sum_{x=1}^{N}\int_{0}^{T}\nabla_N H_{10}\left(\tfrac{x}{N},s\right)\left[\dd J_{10}^{x,x+1}(s)-\dd J_{01}^{x,x+1}(s)\right]\nonumber
    \\&\qquad +
    \sum_{x=1}^{N}\int_{0}^{T}\nabla_N H_{20}\left(\tfrac{x}{N},s\right)\left[\dd J_{20}^{x,x+1}(s)-\dd J_{02}^{x,x+1}(s)\right]\nonumber
    \\&\qquad +
    \sum_{x=1}^{N}\int_{0}^{T}\nabla_N H_{12}\left(\tfrac{x}{N},s\right)\left[\dd J_{12}^{x,x+1}(s)-\dd J_{21}^{x,x+1}(s)\right]\,.
\end{align}}
We use conditions \eqref{requirement-fields}, \eqref{condition-field-1} and \eqref{condition-Field-2}.
Moreover, we denote by $d\eta_{\alpha}^{x}(s)$ the infinitesimal net current of particles of type $\alpha$ crossing the site $x$ up to time $s\in[0,T]$. Therefore, we get 
\begin{align}\label{first-addend}
    &\sum_{x=1}^{N}\left\{\int_{0}^{T}H_{1}\left(\frac{x}{N},s\right)\left[\dd J_{01}^{x,x+1}(s)-\dd J_{10}^{x,x+1}(s)-\dd J_{12}^{x,x+1}(s)+\dd J_{21}^{x,x+1}\right.\right.\nonumber
    \\&\left.\left. \quad\qquad\qquad\qquad\qquad-\dd J_{01}^{x-1,x}(s)+\dd J_{10}^{x-1,x}(s)+\dd J_{12}^{x-1,x}(s)-\dd J_{21}^{x-1,x}(s)\right]\right.\nonumber
    \\+& \left.\int_{0}^{T}H_{2}\left(\frac{x}{N},s\right)\left[\dd J_{02}^{x,x+1}(s)-\dd J_{20}^{x,x+1}(s)-\dd J_{21}^{x,x+1}(s)+\dd J_{12}^{x,x+1}(s)\right.\right.\nonumber
    \\&\left.\left.\quad\qquad\qquad\qquad\qquad-\dd J_{02}^{x-1,x}(s)+\dd J_{20}^{x-1,x}(s)+\dd J_{21}^{x-1,x}(s)-\dd J_{12}^{x-1,x}(s)\right]\right\}\nonumber
    \\=&
    \sum_{x=1}^{N}\left\{\int_{0}^{T}H_{1}\left(\frac{x}{N},s\right)\dd \eta_{1}^{x}(s)+\int_{0}^{T}H_{2}\left(\frac{x}{N},s\right)\dd \eta_{2}^{x}(s)\right\}\nonumber
    \\=&
    N\langle \mu_{1}^{N}(T),H_{1}\left(\cdot,T\right)\rangle+N\langle\mu_{2}^{N}(T),H_{2}\left(\cdot,T\right)\rangle-N\langle \mu_{1}^{N}(0),H_{1}\left(\cdot,0\right)\rangle-N\langle\mu_{2}^{N}(0),H_{2}\left(\cdot,0\right)\rangle\nonumber
    \\-& N\int_{0}^{T}\langle \mu_{1}^{N}(s),\partial_{s}H_{1}(\cdot,s)\rangle \dd s-N\int_{0}^{T}\langle \mu_{1}^{N}(s), \partial_{s}H_{2}(\cdot,s)\rangle \dd s.
\end{align}
where in the first equality we used 
    \begin{align}
        \dd \eta_1^x(s) &= \dd J_{01}^{x,x+1}(s)-\dd J_{10}^{x,x+1}(s)-\dd J_{12}^{x,x+1}(s)+\dd J_{21}^{x,x+1}\nonumber
    \\&\quad\qquad\qquad-\dd J_{01}^{x-1,x}(s)+\dd J_{10}^{x-1,x}(s)+\dd J_{12}^{x-1,x}(s)-\dd J_{21}^{x-1,x}(s),\nonumber\\
        \dd \eta_2^x(s) &=\dd J_{02}^{x,x+1}(s)-\dd J_{20}^{x,x+1}(s)-\dd J_{21}^{x,x+1}(s)+\dd J_{12}^{x,x+1}(s)\nonumber
    \\&\quad\qquad\qquad-\dd J_{02}^{x-1,x}(s)+\dd J_{20}^{x-1,x}(s)+\dd J_{21}^{x-1,x}(s)-\dd J_{12}^{x-1,x}(s),
    \end{align}
and in the last equality of \eqref{first-addend} we have integrated by parts. \\  
To conclude the proof we have to show that \eqref{action-generator-exponential} holds true. By applying the generator \eqref{symmetirc-genrator} we have that 
\begin{align}
     &N^{2}e^{-N\langle\bm\mu_N(s), \bm H(\cdot,s)\rangle}\mathcal{L}e^{N\langle\bm\mu_N(s), \bm H(\cdot,s)\rangle}\nonumber
     \\&=
     N^{2}\sum_{x=1}^{N}\sum_{\alpha,\beta=0}^{2}\eta_{\alpha}^{x}\eta_{\beta}^{x+1}\left(\exp{\left\{N\langle\mu_{1}^{N}(\bm{\eta}_{\alpha,\beta}^{x,x+1}(N^{2}s)),H_{1}(\cdot,s)\rangle+N\langle\mu_{2}^{N}(\bm{\eta}_{\alpha,\beta}^{x,x+1}(N^{2}s)),H_{2}(\cdot,s)\rangle\right\}}\nonumber\right.
     \\&\qquad\qquad \quad\,\cdot\left.\exp{\left\{-N\langle\mu_{1}^{N}(\bm{\eta}(N^{2}s)),H_{1}(\cdot,s)\rangle-N\langle\mu_{2}^{N}(\bm{\eta}(N^{2}s)),H_{2}(\cdot,s)\rangle\right\}}-1\right)\nonumber
     \\&=
     \sum_{x=1}^{N}\sum_{\alpha,\beta=0}^{2}\eta_{\alpha}^{x}\eta_{\beta}^{x+1}\left[\exp{\left\{\nabla_N H_{\alpha\beta}\left(\tfrac{x}{N},s\right)\right\}}-1\right],
\end{align}
where we have used the conditions on the {\color{black}potentials} expressed in equations \eqref{requirement-fields}, \eqref{condition-field-1} and \eqref{condition-Field-2}. 

\end{proof}
\begin{corollary}
Equation \eqref{action-generator-exponential} can be written as
\begin{align}\label{taylor-approx-generator-exponential}
        N^{2}&e^{-N\langle\bm\mu_N(s), \bm H(\cdot,s)\rangle}\mathcal{L}e^{N\langle\bm\mu_N(s), \bm H(\cdot,s)\rangle}\nonumber
        \\&=
        \sum_{x=1}^N\left\{\eta_{1}^{x}\left[\Delta H_{1}\left(\tfrac{x}{N},s\right)+(1-\eta_{1}^{x+1})\left(\nabla H_{1}\left(\tfrac{x}{N},s\right)\right)^{2}\right]
      + 
    \left.\eta_{2}^{x}\left[\Delta H_{2}\left(\tfrac{x}{N},s\right)+(1-\eta_{2}^{x+1})\left(\nabla H_{2}\left(\tfrac{x}{N},s\right)\right)^{2}\right]\right.\right.\nonumber
   \\&\qquad\left.-
   \eta_{1}^{x}\eta_{2}^{x+1}\left[\nabla H_{1}\left(\tfrac{x}{N},s\right)\;\nabla H_2\left(\tfrac{x}{N},s\right)\right]
   -
   \eta_{2}^{x}\eta_{1}^{x+1}\left[\nabla H_{1}\left(\tfrac{x}{N},s\right)\;\nabla H_2\left(\tfrac{x}{N},s\right)\right]\right\} +\mathcal{O}(1)\,.
    \end{align}
   Furthermore, for all $H_{1},H_{2}\in C^{2,1}(\mathbb{T}\times [0,T])$ and for all $N\in \mathbb{N}$ there exists a constant $c>0$ such that 
\begin{equation}\label{estimate-for-expMartingale}
   \frac{dP_{N}^{\bm H,\gamma}}{dP_{N}^{\bm\gamma}}
   \leq \exp{\left\{cN\right\}}\,.
\end{equation}
\end{corollary}
\begin{proof}
First we expand the exponential function on the right hand side of \eqref{action-generator-exponential} by using the Taylor series and we also use \eqref{exclusion-constraint}. Finally, we use conditions \eqref{requirement-fields}, \eqref{condition-field-1} and \eqref{condition-Field-2}, obtaining  
\begin{align}\label{second-addend}
    \sum_{x=1}^{N}&\sum_{\alpha,\beta=0}^{2}\eta_{\alpha}^{x}\eta_{\beta}^{x+1}\left[\exp{\left\{\nabla_N H_{\alpha\beta}\left(\tfrac{x}{N},s\right)\right\}}-1\right]\nonumber\\
    =&N^{2}\sum_{x=1}^{N}\sum_{\alpha\beta=0}^{2}\left\{H_{\alpha\beta}\left(\tfrac{x+1}{N},s\right)-H_{\alpha\beta}\left(\tfrac{x}{N},s\right)
    +\frac{1}{2}\left(H_{\alpha\beta}\left(\tfrac{x+1}{N}\right)-H_{\alpha\beta}\left(\tfrac{x}{N},s\right)\right)^{2}\right\}\eta_{\alpha}^{x}\eta_{\beta}^{x+1}+\mathcal{O}(1)\nonumber
    \\=&\sum_{x=1}^{N}
    \left\{\eta_{1}^{x}\Delta_{N}H_{1}\left(\tfrac{x}{N},s\right)+\eta_{2}^{x}\Delta_{N}H_{2}\left(\tfrac{x}{N},s\right)+\eta_{1}^{x}\left(1-\eta_{1}^{x+1}\right)\left(\nabla H_{1}\left(\tfrac{x}{N},s\right)\right)^{2}\right.\nonumber
    \\&\qquad+
    \left.\eta_{2}^{x}\left(1-\eta_{2}^{x+1}\right)\left(\nabla H_{2}\left(\tfrac{x}{N},s\right)\right)^{2}-\left(\eta_{1}^{x}\eta_{2}^{x+1}+
    \eta_{2}^{x}\eta_{1}^{x+1}\right)\left(\nabla H_{1}\left(\tfrac{x}{N},s\right) \nabla H_{2}\left(\tfrac{x}{N},s\right)\right)\right\}+\mathcal{O}(1)\,.
\end{align}
    The estimate \eqref{estimate-for-expMartingale} follows, since $\eta_{\alpha}^{x}\leq 1$ for all $a\in \{1,2\}$ and for all $x\in \mathbb{T}_{N}$ and because the functions $H_{1}(\cdot,\cdot),H_{2}(\cdot,\cdot)$ belong to the space $C^{2,1}(\mathbb{T}\times [0,T])$.  
\end{proof}
\begin{corollary}
    The super exponential estimate \eqref{supo} holds also for the path space measure $\mathbb{P}_{N}^{\gamma,\bfH}$.
\end{corollary}
\begin{proof}
For any measurable set $A\subset D\left([0,T],\mathcal{M}_{1}\times \mathcal{M}_{1}\right)$ we have the following chain of inequalities
\begin{align}
    \frac{1}{N}\log\mathbb{P}_{N}^{\bm\gamma,\bfH}(A)
    =&\frac{1}{N}\log\E^{\bm\gamma,\bfH}_{N}\left[\mathbbm{1}_{\{A\}}\frac{\dd\mathbb{P}_{N}^{\bm\gamma}}{\dd\mathbb{P}_{N}^{\bm\gamma,\bm{H}}}\frac{\dd \mathbb{P}_{N}^{\bm\gamma,\bfH}}{\dd \mathbb{P}_{N}^{\bm\gamma}}\right]
    =\frac{1}{N}\log \mathbb{E}^{\bm\gamma}_{N}\left[\mathbbm{1}_{\{A\}}\frac{d \mathbb{P}_{N}^{\bm\gamma,\bfH}}{d \mathbb{P}_{N}^{\bm\gamma}}\right]\nonumber
    \\\leq &
    \frac{1}{N}\log \mathbb{E}_N^{\bm\gamma}\left[\mathbbm{1}_{\{A\}}\right]+c\,.
\end{align}
Here we have changed the path space measure from $\mathbb{P}_{N}^{\bm{\gamma},\bm{H}}$ to $\mathbb{P}_{N}^{\bm{\gamma}}$ and we have used the estimate \eqref{estimate-for-expMartingale}. Therefore, by taking the limit  $N\to \infty$ and by using Theorem \ref{superexponential estimate} we have the result.
\end{proof}

\subsection{Exponential tightness}\label{exponential tightness section}

\begin{theorem}[Exponential Tightness]\label{exponential tightness}
    For any $n\in\mathbb{N}$ there exists a compact set $\mathcal{K}_n\subset D([0,T],\mathcal{M}_1\times \mathcal{M}_1)$ such that 
    \begin{equation}
    \uplim_{N\to\infty} \frac{1}{N} \log \p^{\bm\gamma}_N(\bm\mu_N \notin \mathcal{K}_n) = -n.
    \end{equation}
\end{theorem}

With exponential tightness, the large deviation upper bound for closed sets $\mathcal{C} \in D([0,T], \mathcal{M}_1\times \mathcal{M}_1)$ follows from the upper bound for compact sets. Namely, for every $n\in\mathbb{N}$ we have that 
\begin{equation}
\uplim_{N\to\infty} \frac{1}{N} \log \p^{\bm\gamma}_N (\bm \mu_N \in \mathcal{C}) \leq   \uplim_{N\to\infty} \frac{1}{N} \log\left[  \p^{\bm\gamma}_N (\bm \mu_N \in \mathcal{C}\cap \mathcal{K}_n) \vee \p^{\bm\gamma}_N(\bm \mu_N \notin \mathcal{K}_n)\right], 
\end{equation}
    where $\mathcal{C}\cap \mathcal{K}_n$ is a compact set. 
    
We will prove Theorem \ref{exponential tightness} following the same approach as in \cite[Section 10.4]{KipnisLandim}. We start with the following Lemma.
\begin{lemma}\label{continuity}
    For every $\varepsilon>0$ and $\bm G \in C^2(\mathbb{T})\times C^2(\mathbb{T})$ 
    \begin{equation}\label{Arzela Ascoli}
    \lim_{\delta\to\infty} \lim_{N\to\infty} \frac{1}{N}\log \p^{\bm\gamma}_N\left(\sup_{|s-t|<\delta} \left| \left<\bm\mu_N(t), \bm G\right> - \left<\bm\mu_N(s), \bm G\right> \right| \geq \varepsilon \right)= -\infty.
    \end{equation}
\end{lemma}
\begin{proof}
    First note that
    \begin{equation}\label{contained}
    \left\{ \sup_{|s-t|<\delta} \left| \left<\bm\mu_N(t), \bm G\right> - \left<\bm\mu_N(s), \bm G\right> \right| \geq \varepsilon\right\} \subset \bigcup_{k=0}^{[T\delta^{-1}]} \left\{\sup_{k\delta \leq t<(k+1)\delta} \left|\left<\bm\mu_N(t), \bm G\right> - \left<\bm\mu_N(k\delta), \bm G\right> \right|>\frac{\varepsilon}{4}\right\}.
    \end{equation}
    Therefore, we find that
    \begin{align*}
    &\uplim_{N\to\infty} \frac{1}{N}\log \p^{\bm\gamma}_N\left(\sup_{|s-t|<\delta} \left| \left<\bm\mu_N(t), \bm G\right> - \left<\bm\mu_N(s), \bm G\right> \right| \geq \varepsilon \right)\\
     \leq &\uplim_{N\to\infty} \frac{1}{N}\log \sup_{k=0}^{[T\delta^{-1}]} \p^{\bm\gamma}_N\left(\sup_{k\delta \leq t<(k+1)\delta} \left| \left<\bm\mu_N(t), \bm G\right> - \left<\bm\mu_N(k\delta), \bm G\right> \right| \geq \varepsilon \right)\\
     =&\uplim_{N\to\infty} \frac{1}{N}\log  \p^{\bm\gamma}_N\left(\sup_{0\leq t<\delta} \left| \left<\bm\mu_N(t), \bm G\right> - \left<\bm\mu_N(0), \bm G\right> \right| \geq \varepsilon \right),
    \end{align*}
    where we used that $\p^{\bm\gamma}_N$ is an invariant measure for the last equality.
    
    Since we are considering every $\bm G$, we can neglect the absolute value. Furthermore, we have that for any $\lambda>0$
    \begin{align}
         &\p^{\bm\gamma}_N\left(\sup_{0\leq t<\delta} \left<\bm\mu_N(t), \bm G\right> - \left<\bm\mu_N(0), \bm G\right>  \geq \varepsilon \right)\nonumber
         \\=& \p^{\bm\gamma}_N\left(\sup_{0\leq t<\delta}\frac{1}{N}\log Z_{t,N}^{\lambda \bfG}(\bm\mu_N)  + \frac{1}{N} \int_0^te^{-\lambda N\langle\bm\mu_N(s), \bm G\rangle}\left(\partial_{s}+ N^2\mathcal{L}\right) e^{\lambda N\langle \bm\mu_N(s), \bm G\rangle}\dd s \geq \lambda \varepsilon \right)\nonumber\\
        \leq& \p^{\bm\gamma}_N\left(\sup_{0\leq t<\delta}\frac{1}{N}\log Z_{t,N}^{\lambda \bfG}(\bm\mu_N) \geq \tfrac{1}{2}\lambda \varepsilon \right) + \p^{\bm\gamma}_N\left( \sup_{0\leq t<\delta}\frac{1}{N} \int_0^te^{-\lambda N\langle\bm\mu_N(s), \bm G\rangle}\left(\partial_{s}+ N^2\mathcal{L}\right) e^{\lambda N\langle \bm\mu_N(s), \bm G\rangle}\dd s \geq \tfrac{1}{2}\lambda \varepsilon \right).
    \end{align}
    Note that by \eqref{taylor-approx-generator-exponential} and the fact that there is at most one particle per site,  
    \begin{equation}
        \frac{1}{N} \int_0^te^{-\lambda N\langle\bm\mu_N(s), \bm G\rangle}\left(\partial_{s}+ N^2\mathcal{L}\right) e^{\lambda N\langle \bm\mu_N(s), \bm G\rangle}\dd s  = \mathcal{O}(\delta).
    \end{equation}
    Furthermore, by Doob's martingale inequality
    \begin{equation}
    \p^{\bm\gamma}_N\left(\sup_{0\leq t<\delta}\frac{1}{N}\log Z_{t,N}^{\lambda \bfG}(\bm\mu_N)  \tfrac{1}{2}\lambda \varepsilon\right) =\p^{\bm\gamma}_N\left(\sup_{0\leq t<\delta} Z_{t,N}^{\lambda \bfG}(\bm\mu_N) \geq e^{\frac{1}{2}N\lambda \varepsilon}\right) \leq e^{-\tfrac{1}{2}N\lambda \varepsilon}, 
    \end{equation}
    where we used that $Z_{t,N}^{\lambda \bfG}(\bm\mu_N)$ is a mean 1 martingale. Therefore we find that 
    \begin{equation}
    \uplim_{\delta\to 0}\ \uplim_{N\to\infty} \frac{1}{N}\log  \p^{\bm\gamma}_N\left(\sup_{0\leq t<\delta}  \left<\bm\mu_N(t), \bm G\right> - \left<\bm\mu_N(0), \bm G\right>  \geq \varepsilon \right) = -\tfrac{1}{2}\lambda \varepsilon,
    \end{equation}
    and since we took $\lambda>0$ arbitrary this concludes the proof.
\end{proof}

With this Lemma we are able to prove the exponential tightness of the empirical distributions.
\begin{proof}[Proof of Theorem \ref{exponential tightness}]
Consider a countable uniformly dense family  $\{\bfH_k\}_{k\in\mathbb{N}}\subset C^2(\mathbb{T})\times C^2(\mathbb{T})$. Then, for each $\delta>0$, $\epsilon>0$ we define the following set 
\begin{equation}
\mathcal{C}_{k,\delta,\epsilon} = \left\{\bm \mu \in D([0,T], \mathcal{M}_1\times \mathcal{M}_1) ; \sup_{|t-s|\leq \delta} \left|\langle \bm \mu(t), \bfH_k\rangle - \langle \bm\mu(s), \bfH_k\rangle \right|\leq \epsilon \right\}.
\end{equation}
First of all, note that $\mathcal{C}_{k, \delta, \epsilon}$ is closed. Furthermore, by Lemma \ref{continuity} we know that we can find a $\delta=\delta(k,m,n)$ such that 
\begin{equation}
\p^{\bm\gamma}_N (\bm \mu_N \not \in \mathcal{C}_{k,\delta, 1/m}) \leq \exp(-Nnmk)
\end{equation}
for $N$ large enough. 
We then define 
\begin{equation}
\mathcal{K}_n = \bigcap_{k \geq 1, m\geq 1} \mathcal{C}_{k, \delta(k,m,n),1/m}.
\end{equation}
Then we find that 
\begin{equation}
\p^{\bm\gamma}_N (\bm \mu_N \not \in \mathcal{K}_n) \leq \sum_{k\geq 1, m\geq 1}\exp(-Nnmk) \leq C\exp(-Nn)
\end{equation}
where $C>0$ is some constant, and so 
\begin{equation}
\uplim_{N\to\infty} \frac{1}{N}\log \p^{\bm\gamma}_N (\bm \mu_N \not \in \mathcal{K}_n) \leq -n.
\end{equation}

Since $\mathcal{K}_n$ is closed, we now only have to show that $\mathcal{K}_n$ is relatively compact for every $n\in\mathbb{N}$, i.e., we need to show that the following two things holds \cite[Proposition 4.1.2]{KipnisLandim}
\begin{enumerate}
    \item \label{1}$\{\bm\mu(t) ; \bm\mu \in \mathcal{K}_n, t\in[0,T]\}$ is relatively compact in $\mathcal{M}_1\times \mathcal{M}_1$.
    \item \label{2}$\lim_{\delta \to 0} \sup_{\bm \mu \in \mathcal{K}_n} w_\delta(\bm\mu) = 0$ where 
    \begin{equation}
    w_\delta(\bm\mu) := \sup_{|t-s|\leq \delta} \sum_{k=1}^\infty \frac{1}{2^k} \ \frac{|\langle \bm\mu(t), \bfH_k\rangle - \langle \bm \mu(s), \bfH_k \rangle |}{1+|\langle \bm\mu(t), \bfH_k\rangle - \langle \bm \mu(s), \bfH_k \rangle |}=0.
    \end{equation}
\end{enumerate}
Note here that (\ref{1}) is satisfied since $\mathcal{M}_1\times \mathcal{M}_1$ itself is compact, and (\ref{2}) is satisfied by the definition of $\mathcal{K}_n$, hence $\mathcal{K}_n$ is compact.
\end{proof}

\subsection{Proof of the upper bound}\label{upper bound}

\begin{theorem}[Upper bound for compact sets]
    For any compact set $\mathcal{K}\subset D\left([0,T],\mathcal{M}_{1}\times \mathcal{M}_{1}\right)$ we have that 
    \begin{equation}
        \uplim_{N\to\infty}\frac{1}{N}\log \mathbb{P}_{N}^{\bm\gamma}\left(\bm\mu_N\in \mathcal{K}\right)\leq -\inf_{\bfrho\in \mathcal{K}}I_{\bm\gamma}(\bfrho).
    \end{equation}
\end{theorem}
\begin{proof}
For any given $G_{1},G_{2}\in C^{2,1}(\mathbb{T}\times [0,T])$, and $\epsilon>0$, $\delta>0$, we introduce the following
\begin{align}
    B_{\epsilon,N,G_{\alpha},G_{\beta}}^{\delta,\alpha,\beta}:=&\left\{\bm{\eta}(t),0\leq t\leq T\,:\,\left|\frac{1}{N}\sum_{x=1}^{N}\int_{0}^{T}\nabla G_{\alpha}\left(\tfrac{x}{N},s\right)\nabla G_{\beta}\left(\tfrac{x}{N},s\right)\right.\right.\nonumber
    \\\cdot&\left.\left.\left(\eta_{\alpha}^{x}(s)\eta_{\beta}^{x+1}(s)-\left[\frac{1}{2\epsilon N+1}\sum_{|y-x|  \leq N\epsilon}\eta_{\alpha}^{y}\right]\left[\frac{1}{2\epsilon N+1}\sum_{|y-x|\leq N\epsilon}\eta_{\beta}^{y}\right]\right)\dd s\right|\leq \delta \right\}\,.
\end{align}
Moreover, we denote by 
\begin{equation}
    B_{\epsilon,N,\bm{G}}^{\delta}=\cap_{\alpha,\beta=1}^{2}B_{\epsilon,N,G_{\alpha},G_{\beta}}^{\delta,\alpha,\beta}\,.
\end{equation}
By the superexponential estimate, we then have that 
\begin{align}\label{estimate-fake-Martingale}
 \uplim_{N\to\infty}\frac{1}{N}\log\mathbb{P}_{N}^{\bm\gamma}\left(\left\{\bm\mu_N\in \mathcal{K}\right\}\right)\leq \uplim_{\delta\to 0}\;\uplim_{N\to\infty}\frac{1}{N}\log\mathbb{P}_{N}^{\bm\gamma}\left(\left\{\bm\mu_N\in \mathcal{K}\right\}\;\cap\; B_{\epsilon,N,\bm{G}}^{\delta}\right)\,.
 \end{align}
 We now define $q_\epsilon := \frac{1}{2\epsilon}\mathbbm{1}_{\{[-\epsilon,+\epsilon]\}}$ and we introduce the following
 \begin{align}
     \widetilde{Z}_{T,N}^{\bm{G}}(\bm\mu_N \ast \bm q_\epsilon):=
     &\exp{\Big\{ N\langle \left(\bm\mu_N(T)\ast \bm q_{\epsilon}\right),\bfG(\cdot,T)\rangle-N\langle \left(\bm\mu_N(0)\ast \bm q_{\epsilon}\right),\bfG(\cdot,0)\rangle \Big\}}\nonumber
     \\\cdot&
     \exp{\left\{- N\int_0^T \langle (\bm\mu_N(t)\ast \bm q_\epsilon), (\partial_s + \Delta)\bfG(\cdot, t)\rangle \dd t\right\}}\nonumber
     \\\cdot&
     \exp{\left\{- N\int_{0}^{T}\langle \left(\mu_{1,N}(t)\ast q_{\epsilon}\right)\left(1-\left(\mu_{1,N}(t)\ast q_{\epsilon}\right)\right), \left(\nabla G_{1}(\cdot,t)\right)^{2}\rangle \dd t\right\}}\nonumber
     \\\cdot&
     \exp{\left\{- N\int_{0}^{T}\langle \left(\mu_{2,N}(t)\ast q_{\epsilon}\right)\left(1-\left(\mu_{2,N}(t)\ast q_{\epsilon}\right)\right), \left(\nabla G_{2}(\cdot,t)\right)^{2}\rangle \dd t\right\}}\nonumber
     \\\cdot &
     \exp{\left\{- 2N\int_{0}^{t}\langle\left(\mu_{1,N}(t)\ast q_{\epsilon}\right)\left(\mu_{2,N}(t)\ast q_{\epsilon}\right),\nabla G_{1}(\cdot,t)\nabla G_{2}(\cdot,t) \rangle \dd t\right\}}\,.
 \end{align}
By this definition and by using the exponential martingale, for all $N$ and for all $\{\bm{\eta}(t),0\leq t\leq T\} \in B_{\epsilon,N,\bm{G}}^{\delta}$ we have that, 
\begin{equation}
    \widetilde{Z}_{T,N}^{\bm{G}}(\bm\mu_N\ast \bm q_\epsilon)\leq Z_{T,N}^{\bm{G}}(\bm \mu_N) \exp{\left\{N(c(\epsilon)+\delta)\right\}}.
\end{equation}
Using \eqref{estimate-fake-Martingale} and recalling the definition of $h_{\bm\phi}$ in \eqref{static} we then find that 
\begin{align}
    &\uplim_{N\to\infty}\log \mathbb{P}_{N}^{\bm\gamma}\left(\left\{\bm\mu_N\in \mathcal{K}\right\}\;\cap\; B_{\epsilon,N,\bm{G}}^{\delta}\right)\nonumber\\
    &\qquad \qquad \qquad =\uplim_{N\to\infty}\frac{1}{N}\log \mathbb{E}_{\nu^{\bm\gamma}_N}\left[\mathbbm{1}_{\left\{\left\{\bm\mu_N\in \mathcal{K}\right\}\;\cap\; B_{\epsilon,N,\bm{G}}^{\delta}\right\}} \frac{\widetilde{Z}_{T,N}^{\bm{G}}(\bm\mu_N \ast \bm q_\epsilon)}{\widetilde{Z}_{T,N}^{\bm{G}}(\bm\mu_N \ast \bm q_\epsilon)}\cdot \frac{e^{Nh_{\bm\phi}(\bm\mu_N(0);\bm\gamma)}}{e^{Nh_{\bm\phi}(\bm\mu_N(0);\bm\gamma)}}\right]\nonumber
    \\&\qquad \qquad \qquad\leq
    \uplim_{N\to\infty}\frac{1}{N}\log{\mathbb{E}_{\nu^{\bm\gamma}_N}\left[Z_{T,N}^{\bm{G}}(\bm \mu_N)\cdot e^{Nh_{\bm\phi}(\bm\mu_N(0);\bm\gamma)}\right]} + c(\epsilon) + \delta \nonumber
    \\ &\qquad \qquad \qquad- 
    \inf_{\bfrho\in \mathcal{K} }\left\{ h_{\bm\phi}(\bfrho(0);\bm\gamma) + \ell(\bfrho\ast\bm q_{\epsilon};\bfG) - ||\bfG||_{\mathcal{H}(\bfrho\ast \bm q_\epsilon)}\right\}.
\end{align}
Since $Z_{T,N}^{\bm{G}}(\bm \mu_N)$ is a martingale with $Z_{0,N}^{\bm{G}}(\bm \mu_N) = 1$
\begin{equation}
     \mathbb{E}_{\nu^{\bm\gamma}_N}\left[Z_{T,N}^{\bm{G}}(\bm \mu_N)e^{Nh_{\bm\phi}(\bm\mu_N(0);\gamma)}\right]= \mathbb{E}_{\nu^{\bm\gamma}_N}\left[e^{Nh_{\bm\phi}(\bm\mu_N(0);\gamma)}\right] =1\,.
\end{equation}
By taking the limsup for $\epsilon\to 0$ and $\delta\to0$, by optimizing over $\bm{G}$ and over $\bm{\phi}$ and by exchanging the supremum and the infimum (by using the argument of Lemma 11.3 of \cite{varadhan1984large}) we obtain that 
\begin{align}
    \uplim_{\epsilon\to 0}\;\uplim_{\delta \to 0}\;\uplim_{N\to\infty}\frac{1}{N}\log \mathbb{P}_{N}^{\bm\gamma}\left(\left\{\bm\mu_N\in \mathcal{K}\right\}\;\cap\; B_{\epsilon,N,\bm{G}}^{\delta}\right)\leq -\inf_{\bm\rho\in \mathcal{K}}\mathcal{I}_{\bm\gamma}(\bm\rho),
\end{align}
then the Theorem follows. 
\end{proof}

\subsection{Proof of the lower bound}\label{lower bound}
\begin{lemma}\label{explicit static}
    Assume that $h(\bfrho(0);\bm\gamma)<\infty$, then there exists a density $\bm\omega :=\frac{\dd\bfrho(0)}{\dd \lambda}$, with $\lambda$ the Lebesgue measure, and 
    \begin{equation}
    h(\bfrho(0);\bm\gamma) =  \lim_{N\to\infty} \frac{1}{N}\E_{\nu^{\bm\omega}_N}\left[\log\left(\frac{\dd \nu^{\bm\omega}_N}{\dd \nu^{\bm\gamma}_N}\right)\right].
    \end{equation}
\end{lemma}
\begin{proof}
    If $\bfrho(0)$ is not absolutely continuous with respect to the Lebesgue measure, then there exists a $u\in\mathbb{T}$ and $\alpha \in \{1,2\}$ such that $\rho_\alpha(0)(\{u\}) >0$. Then, by choosing a sequence $(\phi_{\alpha,k})_{k\in\mathbb{N}}$ such that $\phi_{\alpha,k}(u)\to\infty$, one can show that $h(\bfrho(0);\bm\gamma) = \infty$. Hence if $h(\bfrho(0);\bm\gamma)<\infty$ we have that $\bm\omega$ exists. The rest of a proof is just a calculation.
    \begin{align*}
    h(\bfrho(0);\bm\gamma) 
    &= \sup_{\bm\phi} \left\{ \sum_{\alpha=0}^2 \langle \omega_\alpha,\phi_\alpha\rangle - \int_{\mathbb{T}}\log\left(\sum_{\alpha=0}^2 \gamma_\alpha(u) e^{\phi_\alpha(u)}\right)\dd u \right\} \\
    &= \sum_{\alpha=0}^2\langle \omega_\alpha,\log\left(\frac{\omega_\alpha}{\gamma_\alpha}\right)\rangle \\
    &=\lim_{N\to\infty} \frac{1}{N}\E_{\nu^{\bm\omega}_N}\left[\log\left(\frac{\dd \nu^{\bm\omega}_N}{\dd \nu^{\bm\gamma}_N}\right)\right] .
    \end{align*}
\end{proof}
\begin{lemma}\label{lemma51}
    Assume that $\mathcal{I}_{0}(\bfrho)<\infty$, then there exists an $\bfH\in \mathcal{H}(\bfrho)$ such that  for all $\bfG$ we have that 
\begin{align}\label{weakSolution-Riesz}
    \ell(\bfrho;\bfG)=\left<\bfG,\bfH\right>_{\mathcal{H}(\bf\rho)}.
\end{align}
Moreover, the following holds
\begin{align}
    \mathcal{I}_{0}(\bfrho)=\tfrac{1}{2}||\bfH||_{\mathcal{H}(\bfrho)}^2.
\end{align}
\end{lemma}
\begin{remark}
    Observe that if \eqref{weakSolution-Riesz} holds for all $\bfG$, then $\bfrho$ satisfies the equations $\eqref{HD-equations LDP}$ in the sense of distributions. This will be used in the proof of the large deviation lower bound. Indeed, by choosing a non-typical trajectory, we can find an $\bm{H}$ that makes it typical, i.e. that makes it solve the weakly asymmetric  hydrodynamic equations.
\end{remark}
\begin{proof}
By definition, we have that 
    \begin{align}
        \mathcal{I}_0(\bfrho)\geq \lambda \ell(\bfrho;\bfG) - \frac{1}{2}\lambda^2 ||\bfG||_{\mathcal{H}(\bfrho)}^2
    \end{align}
    for any $\lambda>0$. 
    Optimizing over $\lambda$ we have that 
    \begin{align}
        \lambda^{*}=\frac{\ell(\bfrho;\bfG)}{||\bfG||_{\mathcal{H}(\bfrho)}^2},
    \end{align}
    and so  
    \begin{align}
        \ell(\bfrho;\bfG)^2\leq 2 \mathcal{I}_{0}(\bfrho)||\bfG||_{\mathcal{H}(\bfrho)}^2.
    \end{align}
    This means that the linear functional $\ell(\bfrho;\cdot)$  is bounded in the Hilbert space $\mathcal{H}(\bfrho)$ and so, by the Riesz representation Theorem, there exists an $\bfH \in \mathcal{H}(\bfrho)$ such that  \eqref{weakSolution-Riesz} holds for all $\bfG$.

Using this, we find that 
\begin{align}\label{norm rate function}
    \mathcal{I}_{0}(\bfrho)
    &=\sup_{\bfG}\left\{\langle\bm{G},\bm{H}\rangle_{\mathcal{H}(\bfrho)}-\tfrac{1}{2}||\bm{G}||_{\mathcal{H}(\bfrho)}^2\right\}\nonumber\\
     &=\sup_{\bfG}\left\{\tfrac{1}{2}||\bfH||_{\mathcal{H}(\bfrho)}^2-\tfrac{1}{2}||\bm{H}-\bm{G}||_{\mathcal{H}(\bfrho)}^2\right\}\nonumber\\\
    &=\tfrac{1}{2}||\bm{H}||_{\mathcal{H}(\bfrho)}^2,
    \end{align}
    which concludes the proof.
\end{proof}
\begin{theorem}\label{them-neighb}
   Fix $\bfrho \in D\left([0,T],\mathcal{M}_{1}\times \mathcal{M}_{1}\right) $, then for any open neighborhood $\mathcal{O}$ around $\bfrho$ we have that  
    \begin{align*}
        \lowlim_{N\to \infty}\;\frac{1}{N}\log \mathbb{P}_{N}^{\bm\gamma}\left(\bm\mu_N\in \mathcal{O}\right)\geq -\mathcal{I}_{\bm\gamma}(\bfrho)
    \end{align*}
\end{theorem}
\begin{proof}
If $\calI_{\bm\gamma}(\bfrho)=\infty$ then the result is immediate, hence we can assume that $\calI_{\bm\gamma}(\bfrho)<\infty$. Therefore, by Lemma \ref{lemma51}, there exists an $\bfH\in\mathcal{H}(\bfrho)$ such that $\ref{HD-equations LDP}$ holds weakly. Fix this $\bfH$ and recall that $\dd\bm\rho(0) = \bm\omega \dd\lambda$, then we denote
\[
\frac{\dd\p^{\bm\omega,\bfH}_{N}}{\dd\p^{\bm\gamma}_N} = \frac{\dd\nu^{\bm\omega}_N}{\dd\nu^{\bm\gamma}_N} \frac{\dd\p^{\bm\gamma,\bfH}_N}{\dd\p^{\bm\gamma}_N}.
\]
We then have that  
\begin{align}
    \mathbb{P}_{N}^{\bm{\gamma}}\left(\bm\mu_N\in \mathcal{O}\right)
    =\E^{\bm\gamma}_{N}\left[\mathbbm{1}_{\{\bm\mu_N\in \mathcal{O}\}}\frac{\dd\p^{\bm\omega,\bfH}_{N}}{\dd\p^{\bm\gamma}_N}\frac{\dd\p^{\bm\gamma}_{N}}{\dd\p^{\bm\omega,\bfH}_N}\right]
    =\E^{\bm\omega,\bfH}_{N}\left[\mathbbm{1}_{\{\bm\mu_N\in \mathcal{O}\}}\frac{\dd\p^{\bm\gamma}_{N}}{\dd\p^{\bm\omega,\bfH}_N}\right].
\end{align}
From Theorem \ref{hydro wasep} it follows that $\bfrho$ is  the typical trajectory of the new dynamics,  and so
\begin{align}\label{typical-path}
    \lim_{N\to \infty}\p^{\bm\omega,\bfH}_{N}\left(\bm\mu_N\in \mathcal{O}\right)=1.
\end{align}
Using this and Jensen inequality, we have that 
\begin{align}
\lowlim_{N\to\infty} \frac{1}{N}\log \p^{\bm\gamma}_N(\bm \mu_N \in \mathcal{O})  
    &=\lowlim_{N\to \infty}\frac{1}{N}\log \E^{\bm\omega,\bfH}_N\left[\frac{\dd\p^{\bm\gamma}_{N}}{\dd\p^{\bm\omega,\bfH}_N}\right]\nonumber
    \\ 
    &\geq \lowlim_{N\to \infty}\frac{1}{N} \E^{\bm\omega,\bfH}_N\left[\log\left(\frac{\dd\p^{\bm\gamma}_{N}}{\dd\p^{\bm\omega,\bfH}_N}\right)\right]\nonumber
    \\ 
    &= \lowlim_{N\to\infty} \frac{1}{N}\E^{\bm\omega,\bfH}_N\left[\log\left(Z_{T,N}^{\bm{H}}(\bm \mu_N)\cdot \frac{\dd\nu^{\bm\gamma}_{N}}{\dd\nu^{\bm\omega}_N}\right)\right]\nonumber\\
    &=-\left[\ell(\bfrho,\bfH)-\tfrac{1}{2}||\bfH||_{\mathcal{H}(\bfrho)}^2+h(\bfrho(0);\bm{\gamma})\right].
\end{align}
Lastly, by Lemma \ref{lemma51}, we know that  $\ell(\bfrho,\bfH) = ||\bfH||_{\mathcal{H}(\bfrho)}^2$ and $\calI_0(\bfrho) = \tfrac{1}{2}||\bfH||_{\mathcal{H}(\bfrho)}^2$, hence we indeed find that 
\begin{equation}
\lowlim_{N\to\infty} \frac{1}{N}\log \p^{\bm\gamma}_N(\bm \mu_N \in \mathcal{O}) \geq -\left[\tfrac{1}{2}||\bfH||_{\mathcal{H}(\bfrho)}^2 + h(\bfrho(0);\bm{\gamma})\right] = \calI_{\bm\gamma}(\bfrho).
\end{equation}
\end{proof}
\begin{theorem}
    For any open set $\mathcal{O}\subset D\left([0,T],\mathcal{M}_{1}\times \mathcal{M}_{1}\right)$ we have that 
    \begin{align}
           \lowlim_{N\to \infty}\;\frac{1}{N}\log \mathbb{P}_{N}^{\bm\gamma}\left(\bm \mu_N\in \mathcal{O}\right)\geq -\inf_{\bfrho\in \mathcal{O}}\mathcal{I}_{\bm\gamma}(\bfrho)\,.
    \end{align}
\end{theorem}
\begin{proof}
The proof is a straightforward consequence of Theorem \ref{them-neighb}.
\end{proof}
\begin{remark}\label{recover 2}
In parallel with  Remark \ref{recover remark}, by choosing the same potential $H_1=H_2=H$ we can recover the large deviation rate function of the dynamics of the single species SEP as given in \cite{KOV}. Namely, by putting  $\varrho = \rho_1+\rho_2$ the rate function $\mathcal{I}_0(\bm\rho)$ from \eqref{norm rate function} becomes a function of $\varrho$ only, i.e., 
\begin{align*}
\mathcal{I}_0(\bm\rho) = \frac{1}{2} ||\bm H||_{\mathcal{H}(\bm\rho)}^2 
= \int_0^T \langle (\varrho(t)(1-\varrho( t)), \big(\nabla H(\cdot, t)\big)^2\rangle \dd t.
\end{align*} 
\end{remark}
\begin{remark}
  In parallel with Remark \ref{generalization remark HD}, the large deviation result reported in this section can be generalized to an arbitrary number of species, namely $\alpha\in\{0,1,\ldots,n\}$. In this case, we consider a $n$-dimensional vector of densities denoted by $\bm{\rho}=(\rho_{1},\ldots,\rho_{n})$. Moreover, we consider $n$-potentials denoted by $H_{\alpha}$, that we list in the vector $\bm{H}=(H_{1},\ldots,H_{n})$.  Therefore, the large deviation functional reads 
    \begin{align}
        \mathcal{I}^{(n)}(\bm{\rho})=\mathcal{I}^{(n)}_{0}(\bm{\rho})+h^{(n)}(\bm{\rho}(0);\bm{\gamma})\,.
    \end{align}
    Here $h^{(n)}(\bm{\rho}(0);\bm{\gamma})$ is the relative entropy between the multinomial distributions with densities corresponding to $\bm{\rho}$ evaluated at time $t=0$ and the original starting density given by $\bm{\gamma}=(\gamma_{1},\ldots,\gamma_{n})$. 
    Moreover, we have that 
    \begin{align}
        \mathcal{I}^{(n)}_{0}(\bm{\rho})=\frac{1}{2}\lVert \bm{H}\rVert_{\mathcal{H}(\bm{\rho})}^2,
    \end{align}
    where the norm is given by 
    \begin{align}
    \lVert \bfH\rVert_{\mathcal{H}(\bm{\rho})}^2&=2\sum_{\alpha=1}^{n}\int_{0}^{T}\langle \rho_{\alpha}(t)(1-\rho_{\alpha}(t)),(\nabla H_{\alpha}(\cdot,t))^{2}\rangle \dd t\nonumber
    \\&-2\sum_{\alpha=1}^{n}\sum_{\beta\neq \alpha}\int_{0}^{T}\langle \rho_{\alpha}(t)\rho_{\beta}(t),\nabla H_{\alpha}(\cdot,t)\nabla H_{\beta}(\cdot,t)\rangle \dd t\,.
    \end{align}
    For any $k\leq n$ we there is a relation between the large deviation rate function of the $n$-species model $\mathcal{I}^{(n)}_0$ and of the $k$-species model $\mathcal{I}^{(k)}_0$. Namely, for any partition $\{A_1,...,A_k\}$ of the set $\{1,...,n\}$, by choosing the same potentials within the partitions, i.e., $H_j = H_\ell$ for every $j \in A_\ell$, we find that 
    \begin{equation}
        \mathcal{I}^{(n)}(\rho_1, ..., \rho_n) = \mathcal{I}^{(k)}_0 (\tilde{\rho}_1,...,\tilde{\rho}_k)
    \end{equation}
    where $\tilde{\rho}_\ell = \sum_{j\in A_\ell} \rho_j$. This generalizes the result in Remark \ref{recover 2}
\end{remark}

\section{Proof of the Hydrodynamic limit of the Weakly-asymmetric process} \label{Appendix B}

We consider the  Dynkin martingale 
\begin{align}
     M_{N}^{\bfG}(t) 
    &= \langle\bm\mu_{N}(t),\bm G(\cdot,t)\rangle -\langle\bm\mu_{N}(0),\bm G(\cdot,0)\rangle -\int_{0}^{t}N^{2}(\mathcal{L}^{\bfH}+\partial_s)\langle \bm\mu_{N}(s),\bm G(\cdot,s)\rangle \dd s\nonumber\\
    &= M_{1,N}^{G_1}(t) + M_{2,N}^{G_2}(t)
\end{align}
where
\begin{align}
    M_{\alpha,N}^{G}(t)=\langle\mu_{\alpha,N}(t),G(\cdot,t)\rangle -\langle\mu_{\alpha,N}(0),G(\cdot,0)\rangle -\int_{0}^{t}N^{2}(\mathcal{L}^{\bfH}+\partial_s) \langle \mu_{\alpha,N}(s),G(\cdot,s)\rangle \dd s.
\end{align}
We see that we need to apply the generator $\mathcal{L}^{\bfH}$ to the density field $\langle \mu_{\alpha,N}(s),G(\cdot,s)\rangle$. This can be derived from the effect of the generator applied to the function $f(\bm\eta) = \eta_\alpha^x$. We start in the case of $\alpha=1$. If we look at $\mathcal{L}^{\bfH} \eta_1^x$ we get a positive (resp. negative) contribution of the rates where a particle of type 1 is added (resp. subtracted) at position $x$, i.e., 

\begin{align*}
\mathcal{L}^{\bfH} \eta_1^x  
&=  c^{\bfH,0,1}_{(x,x+1)}(s) + c^{\bfH,2,1}_{(x,x+1)}(s) + c^{\bfH,1,0}_{(x-1,x)}(s) + c^{\bfH,1,2}_{(x-1,x)}(s)\\
&\qquad \qquad - c^{\bfH,1,0}_{(x,x+1)}(s) - c^{\bfH,1,2}_{(x,x+1)}(s) - c^{\bfH,0,1}_{(x-1,x)}(s) - c^{\bfH,2,1}_{(x-1,x)}(s).
\end{align*}
Using that 
\begin{equation}
c^{\bfH,\alpha\beta}_{(x,x+1)}(s) = \exp\Big(\nabla_NH_{\alpha\beta}(\tfrac{x}{N},s)\Big)\eta_\alpha^{x}\eta_\beta^{x+1}  = \left(1+\nabla_NH_{\alpha\beta}(\tfrac{x}{N},s) + \left(\nabla_NH_{\alpha\beta}(\tfrac{x}{N},s)\right)^2\right)\eta_\alpha^x \eta_\beta^{x+1} + \mathcal{O}(\tfrac{1}{N^3}),
\end{equation}
we find that 
\begin{align*}
    \mathcal{L}^{\bfH} \eta_1^x  
    &= \left(1+ \nabla_NH_{01}(\tfrac{x}{N},s)\right)\eta_0^x \eta_1^{x+1}
    +\left(1+ \nabla_NH_{21}(\tfrac{x}{N},s)\right)\eta_2^x \eta_1^{x+1}\\
    &\qquad \qquad 
    +\left(1+ \nabla_NH_{10}(\tfrac{x-1}{N},s)\right)\eta_1^{x-1} \eta_0^{x}
    +\left(1+ \nabla_NH_{12}(\tfrac{x-1}{N},s)\right)\eta_1^{x-1} \eta_2^{x}\\
    &\qquad \qquad 
    -\left(1+ \nabla_NH_{10}(\tfrac{x}{N},s)\right)\eta_1^x \eta_0^{x+1}
    -\left(1+ \nabla_NH_{12}(\tfrac{x}{N},s)\right)\eta_1^x \eta_2^{x+1}\\
    &\qquad \qquad 
    -\left(1+ \nabla_NH_{01}(\tfrac{x-1}{N},s)\right)\eta_0^{x-1} \eta_1^{x}
    -\left(1+ \nabla_NH_{21}(\tfrac{x-1}{N},s)\right)\eta_2^{x-1} \eta_1^{x} + R(N,x,s)
\end{align*}
with $R(N,x,s)$ a remainder term which we will show vanishes once combined with the test function $G$ as we send $N\to\infty$.

First we will focus on the terms that do not depend on $\bfH$ in the above equation. Using the fact that $\eta_0 = 1 - \eta_1-\eta_2$, after some calculation we find that  
\begin{align*}
    \eta_0^x \eta_1^{x+1}
    +\eta_2^x \eta_1^{x+1}
    +\eta_1^{x-1} \eta_0^{x}
    +\eta_1^{x-1} \eta_2^{x}
    -\eta_1^x \eta_0^{x+1}
    -\eta_1^x \eta_2^{x+1}
    -\eta_0^{x-1} \eta_1^{x}
    -\eta_2^{x-1} \eta_1^{x}\qquad \qquad \\
    = \eta_1^{x+1} + \eta_1^{x-1} - 2\eta_1^x,
\end{align*}
i.e., we recover the discrete Laplacian of $\eta_1^x$. 

For the terms depending on the {\color{black}potential} $H_{10} = -H_{01}$ we then have that 
\begin{align*}
    \left(\eta_1^{x-1} \eta_0^{x}+ \eta_0^{x-1}\eta_1^x\right)\nabla_NH_{10}(\tfrac{x-1}{N},s)
    -\left(\eta_0^x \eta_1^{x+1}  + \eta_1^x \eta_0^{x+1}\right)\nabla_NH_{10}(\tfrac{x}{N},s).
\end{align*}
and the terms depending on the {\color{black}potential} $H_{12}=-H_{21}$
\begin{equation}
  \left(\eta_1^{x-1}\eta_2^x + \eta_2^{x-1}\eta_1^x \right)\nabla_NH_{12}(\tfrac{x-1}{N},s)
  -\left(\eta_2^x\eta_1^{x+1} + \eta_1^x\eta_2^{x+1} \right)\nabla_NH_{12}(\tfrac{x}{N},s).
\end{equation}

With these calculations, we then find that 
\begin{align*}
N^2&\mathcal{L}^{\bfH}\langle \mu_{1,N}(s),G(\cdot,s)\rangle\\ 
&= N\sum_{x=1}^N (\eta_1^{x+1} + \eta_1^{x-1} - 2\eta_1^x)G(\tfrac{x}{N},s)\\
&\qquad + N\sum_{x=1}^N \left(\left(\eta_1^{x-1} \eta_0^{x}+ \eta_0^{x-1}\eta_1^x\right)\nabla_NH_{10}(\tfrac{x-1}{N},s)
    -\left(\eta_0^x \eta_1^{x+1}  + \eta_1^x \eta_0^{x+1}\right)\nabla_NH_{10}(\tfrac{x}{N},s)\right) G(\tfrac{x}{N},s)\\
    &\qquad +N\sum_{x=1}^N \left( \left(\eta_1^{x-1}\eta_2^x + \eta_2^{x-1}\eta_1^x \right)\nabla_NH_{12}(\tfrac{x-1}{N},s)-\left(\eta_2^x\eta_1^{x+1} + \eta_1^x\eta_2^{x+1} \right)\nabla_NH_{12}(\tfrac{x}{N},s)\right) G(\tfrac{x}{N},s)\\
    &\qquad + N\sum_{x=1}^N R(N,x,s)G(\tfrac{x}{N},s).
\end{align*}
where by reordering the terms, we have
\begin{align*}
   N^2\mathcal{L}^{\bfH}\langle \mu_{1,N}(s),G(\cdot,s)\rangle &=N\sum_{x=1}^N \eta_1^x \Delta_NG(\tfrac{x}{N},s) + N\sum_{x=1}^N \left(\eta_0^x \eta_1^{x+1}  + \eta_1^x \eta_0^{x+1}\right)\nabla_NH_{10}(\tfrac{x}{N},s)\nabla_N G(\tfrac{x}{N},s)\\
   &\qquad + N\sum_{x=1}^N \left(\eta_2^x \eta_1^{x+1}  + \eta_1^x \eta_2^{x+1}\right)\nabla_NH_{12}(\tfrac{x}{N},s)\nabla_N G(\tfrac{x}{N},s)+ N\sum_{x=1}^N R(N,x,s)G(\tfrac{x}{N},s)
\end{align*}

The remainder term $R(N,x,s)$ is given by 
\begin{align*}
R(N,x,s) 
&= \left(\eta_0^x\eta_1^{x+1} - \eta_1^x\eta_0^{x+1}\right)\Big(\nabla_NH_{10}(\tfrac{x}{N},s)\Big)^2 
- \left(\eta_0^{x-1}\eta_1^x - \eta_1^{x-1}\eta_0^x\right)\left(\nabla_NH_{10}(\tfrac{x-1}{N},s)\right)^2\\
&\qquad +\left(\eta_2^x\eta_1^{x+1} - \eta_1^x\eta_2^{x+1}\right)\Big(\nabla_NH_{12}(\tfrac{x}{N},s)\Big)^2 
- \left(\eta_2^{x-1}\eta_1^x - \eta_1^{x-1}\eta_2^x\right)\left(\nabla_NH_{12}(\tfrac{x-1}{N},s)\right)^2 + \mathcal{O}(\tfrac{1}{N^3})
\end{align*}
when combined with a test function, we see that 
\begin{align*}
    N\sum_{x=1}^NR(N,x,s)G(\tfrac{x}{N},s) 
    &= -N\sum_{x=1}^N \left(\eta_0^x\eta_1^{x+1} - \eta_1^x\eta_0^{x+1}\right)\Big(\nabla_NH_{10}(\tfrac{x}{N},s)\Big)^2 \nabla_NG(\tfrac{x}{N},s)\\
    &\qquad \qquad - N\sum_{x=1}^N \left(\eta_2^x\eta_1^{x+1} - \eta_1^x\eta_2^{x+1}\right)\Big(\nabla_NH_{12}(\tfrac{x}{N},s)\Big)^2 \nabla_NG(\tfrac{x}{N},s) + \mathcal{O}(\tfrac{1}{N})\,.
\end{align*}
Note that this vanishes as $N\to\infty$ since the discrete derivative $\nabla_N$ is an operator of order $\tfrac{1}{N}$.

Thus, all in all:
\begin{align}\label{action-generator-1}
     N^2\mathcal{L}^{\bfH}\langle \mu_{1,N}(s),G(\cdot,s)\rangle 
    &=\frac{1}{N}\sum_{x\in \mathbb{Z}}\eta_{1}^{x}\Delta  G\left(\tfrac{x}{N}\right)\nonumber
    +\frac{1}{N}\sum_{x=1}^N \left(\eta_0^x \eta_1^{x+1}  + \eta_1^x \eta_0^{x+1}\right)\nabla H_{10}(\tfrac{x}{N},s)\nabla G(\tfrac{x}{N},s)\nonumber\\
   &\qquad + \frac{1}{N}\sum_{x=1}^N \left(\eta_2^x \eta_1^{x+1}  + \eta_1^x \eta_2^{x+1}\right)\nabla H_{12}(\tfrac{x}{N},s)\nabla G(\tfrac{x}{N},s) + \mathcal{O}(\tfrac{1}{N})
\end{align}

By choosing: $\phi(\bm{\eta})=\eta_{1}^{x}\eta_{1}^{x+1}$ and $\phi(\bm{\eta})=\eta_{1}^{x}\eta_{2}^{x+1}$ we now use the superexponential estimate in Theorem \ref{superexponential estimate} twice and we replace 
\begin{align}\label{replace-HD-fast}
    \frac{2}{N}\sum_{x\in\mathbb{Z}}\eta_{1}^{x}\eta_{1}^{x+1}\;\longrightarrow\; \frac{2}{N}\sum_{x\in\mathbb{Z}}\left(\frac{1}{2N\epsilon+1}\sum_{|x-y|\leq \epsilon N}\eta_{1}^{y}\right)\left(\frac{1}{2N\epsilon+1}\sum_{|x-y|\leq \epsilon N}\eta_{1}^{y}\right)
\end{align}
and 
\begin{align}\label{replace-HD-fast-II}
    \frac{2}{N}\sum_{x\in\mathbb{Z}}\eta_{1}^{x}\eta_{2}^{x+1}\;\longrightarrow\; \frac{2}{N}\sum_{x\in\mathbb{Z}}\left(\frac{1}{2N\epsilon+1}\sum_{|x-y|\leq \epsilon N}\eta_{1}^{y}\right)\left(\frac{1}{2N\epsilon+1}\sum_{|x-y|\leq \epsilon N}\eta_{2}^{y}\right)\,.
\end{align}
Indeed, to prove equation \eqref{replace-HD-fast} one writes that, for all $G,H_{10}\in C^{2,1}(\mathbb{T}\times [0,T])$ and for all $a>0$, there exists $\epsilon_{0}>0$ such that for all $\bm{\eta}\in \Omega$ and for all $\epsilon<\epsilon_{0}$ we have that 
\begin{align}
  & \left| \frac{1}{N}\int_{0}^{T}\sum_{x=0}^{N}\nabla G\left(\frac{x}{N},s\right)\nabla H_{10}\left(\frac{x}{N},s\right)\eta_{1}^{x}\eta_{1}^{x+1}\dd s\right.\nonumber
   \\&\left.-\frac{1}{N}\int_{0}^{T}\nabla G\left(\frac{x}{N},s\right)\nabla H_{10}\left(\frac{x}{N},s\right)\left(\frac{1}{2\epsilon N+1}\sum_{|x-y|\leq \epsilon N}\eta_{1}^{y}\eta_{1}^{y+1}\right)\dd s\right| \leq a\,.
\end{align}
Therefore, using the superexponential estimate of Theorem \ref{superexponential estimate} we have that 
\begin{align}
    \uplim_{\epsilon\to 0}\uplim_{N\to \infty}\frac{1}{N}\log \mathbb{P}_{N}^{\bm{\gamma},\bm{H}}&\left(\left| \frac{1}{N}\int_{0}^{T}\sum_{x=0}^{N}\nabla G\left(\frac{x}{N},s\right)\nabla H_{10}\left(\frac{x}{N},s\right)\eta_{1}^{x}\eta_{1}^{x+1}\dd s\right.\right.\nonumber
   \\&\quad -\left.\left.\frac{1}{N}\int_{0}^{T}\nabla G\left(\frac{x}{N},s\right)\nabla H_{10}\left(\frac{x}{N},s\right)\right.\right.\nonumber
   \\&
   \left.\left.\quad \cdot\frac{2}{N}\sum_{x\in\mathbb{Z}}\left(\frac{1}{2N\epsilon+1}\sum_{|x-y|\leq \epsilon N}\eta_{1}^{y}\right)\left(\frac{1}{2N\epsilon+1}\sum_{|x-y|\leq \epsilon N}\eta_{1}^{y}\right)\dd s\right| \geq a\right)=-\infty\,.
\end{align}
With the same argument one can prove \eqref{replace-HD-fast-II} as well. 
Moreover, using $q_\epsilon = \frac{1}{2\epsilon}\mathbbm{1}_{\{[-\epsilon,+\epsilon]\}}$, then we can write  
\begin{align}\label{conv1}
    \frac{1}{N}\sum_{x\in\mathbb{Z}}\left(\frac{1}{2N\epsilon+1}\sum_{|x-y|\leq \epsilon N}\eta_{1}^{y}\right)\left(\frac{1}{2N\epsilon+1}\sum_{|x-y|\leq \epsilon N}\eta_{1}^{y}\right)=\left(\mu_{1,N}(s)\ast q_\epsilon \right) \,\left(\mu_{1,N}(s)\ast q_\epsilon\right) 
\end{align}
and 
\begin{align}\label{conv2}
    \frac{1}{N}\sum_{x\in\mathbb{Z}}\left(\frac{1}{2N\epsilon+1}\sum_{|x-y|\leq \epsilon N}\eta_{1}^{y}\right)\left(\frac{1}{2N\epsilon+1}\sum_{|x-y|\leq \epsilon N}\eta_{2}^{y}\right)=\left(\mu_{1,N}(s)\ast q_\epsilon\right) \,\left(\mu_{2,N}(s)\ast q_\epsilon \right)
\end{align}
where $\ast$ is the convolution. Combining \eqref{action-generator-1} with equations \eqref{conv1} and \eqref{conv2}, we have that the Dynkin martingale $M_{1,N}^{G}(t)$ is written as a function of the empirical density, namely
\begin{align}\label{martingale-final-form}
M_{1,N}^G(t) 
&= \langle\mu_{1,N}(t),G(\cdot,t)\rangle -\langle\mu_{1,N}(0),G(\cdot,0)\rangle -\int_{0}^{t} \langle \mu_{1,N}(s),(\partial_s+\Delta) G(\cdot,s)\rangle \dd s\nonumber\\
&\qquad \qquad -2 \int_0^t \left\langle \left(\mu_{1,N}(s)\ast q_\epsilon \right)\Big(1-\left(\mu_{1,N}(s)\ast q_\epsilon \right)-\left(\mu_{2,N}(s)\ast q_\epsilon \right)\Big), \nabla H_{10}(\cdot, s) \nabla G(\cdot,s)\right\rangle \dd s\nonumber\\
&\qquad \qquad -2 \int_0^t \left\langle \left(\mu_{1,N}(s)\ast q_\epsilon \right)\left(\mu_{2,N}(s)\ast q_\epsilon \right), \nabla H_{12}(\cdot,s)\nabla G(\cdot,s) \right\rangle \dd s + R(\epsilon,N)
\end{align}
where the remainder term $R(\epsilon,N)$ goes to zero in probability as $N\to\infty$ and $\epsilon\to 0$. A similar result can be found for $M_{2,N}^G(t)$. 

Now we show that the martingale $ M_{N}^{\bm G}(t)$ vanishes as $N\to\infty$. 
The quadratic variation is computed by the Carré-du-Champ formula as
\begin{align}
    \Gamma_{N,t}^{\bm{G}}=&N^2\mathcal{L}^{\bm H}\langle \bm{\mu}_{N}(t),\bm{G}(\cdot,t)\rangle^{2}-2\langle \bm{\mu}_{N}(t),\bm{G}(\cdot,t)\rangle N^2\mathcal{L}^{\bm H}\langle \bm{\mu}_{N}(t),\bm{G}(\cdot,t)\rangle\nonumber
    \\=&
    \sum_{x=1}^{N}\sum_{\alpha,\beta=0}^{2}c^{\bm{H},\alpha,\beta}_{(x,x+1)}(t)\left[\langle \bm{\mu}_{N}(\bm{\eta}_{\alpha,\beta}^{x,x+1}(N^{2}t)),\bm{G}(\cdot,t)\rangle-\langle \bm{\mu}_{N}(\bm{\eta}(N^{2}t),\bm{G}(\cdot,t)\rangle\right]^{2}\nonumber
    \\=&
    \sum_{x=1}^{N}\left(c^{\bm{H},0,1}_{(x,x+1)}(t)+c^{\bm{H},1,0}_{(x,x+1)}(t)\right)\left(\nabla_N G_{1}\left(\tfrac{x}{N},t\right)\right)^{2}+\sum_{x=1}^{N}\left(c^{\bm{H},0,2}_{(x,x+1)}(t)+c^{\bm{H},2,0}_{(x,x+1)}(t)\right)\left(\nabla_N G_{2}\left(\tfrac{x}{N},t\right)\right)^{2}\nonumber
    \\&\qquad + \sum_{x=1}^{N}\left(c^{\bm{H},1,2}_{(x,x+1)}(t)+c^{\bm{H},2,1}_{(x,x+1)}(t)\right)\left(\nabla_N G_{1}\left(\tfrac{x}{N},t\right)-\nabla_N G_{2}\left(\tfrac{x}{N},t\right)\right)^{2},
\end{align}
which is of order $1/N$ and goes to $0$ as $N\to \infty$. This implies that for all $\delta>0$
\begin{align}\label{vanishing-martingale}
    \lim_{N\to \infty}\mathbb{P}_{N}^{\bm\gamma,\bfH}\left(\sup_{t\in [0,T]}|M_{N}^{\bm G}(t)|\geq \delta\right)&\leq \lim_{N\to \infty}\frac{1}{\delta^{2}}\mathbb{E}_{N}^{\bm{\gamma},\bm{H}}\left[\sup_{t\in [0,T]}|M_{N}^{\bm G}(t)|^{2}\right]\nonumber
    \\&\leq\lim_{N\to \infty}\frac{4}{\delta^{2}}\mathbb{E}_{N}^{\bm{\gamma},\bm{H}}\left[|M_{N}^{\bm G}(T)|^{2}\right]\nonumber
    \\&=
    \lim_{N\to \infty}\frac{4}{\delta^{2}}\mathbb{E}_{N}^{\bm{\gamma},\bm{H}}\left[\int_{0}^{T}\Gamma_{N,s}^{\bm{G}}\dd s \right]=0\,.
\end{align}
Next, one can show tightness of the sequence of random variables $ \bm{\mu}_{N}$ is tight by using \eqref{vanishing-martingale}. The argument is standard and we refer to \cite{KipnisLandim,seppalainen}. By tightness we have the existence of convergent subsequences, and by combining this with \eqref{martingale-final-form} and \eqref{vanishing-martingale} we observe that these convergent subsequences are concentrated on the set of trajectories $\bm\rho$ such that for all $\delta>0$ there exists an $\hat{\epsilon}$ such that for all $\epsilon\leq \hat{\epsilon}$ and for all $t\in [0,T]$ we have that
\begin{align}
    &\left|\langle \bm{\rho}(t),\bm{G}(\cdot,t)\rangle-\langle \bm{\rho}(0),\bm{G}(\cdot,0)\rangle-\int_{0}^{t}\langle\bm{\rho}(s),\left(\partial_{s}+\Delta \right),\bm{G}(\cdot,s)\rangle\dd s\right.\nonumber
    \\ & \qquad +\left. \int_{0}^{t}\langle\left(\rho_{1}(t)\ast q_{\epsilon}\right)(1-\left(\rho_{1}(t)\ast q_{\epsilon}\right)-\left(\rho_{2}(t)\ast q_{\epsilon}\right)), \nabla G_{1}(\cdot,s)H_{10}(\cdot,s)\rangle\dd s\right.\nonumber 
    \\&
    \left.\qquad +\int_{0}^{t}\langle\left(\rho_{2}(t)\ast q_{\epsilon}\right)(1-\left(\rho_{1}(t)\ast q_{\epsilon}\right)-\left(\rho_{2}(t)\ast q_{\epsilon}\right)), \nabla G_{2}(\cdot,s)\nabla H_{20}(\cdot,s)\rangle\dd s \right.\nonumber
    \\&\qquad - \left.\int_{0}^{t}\langle\left(\rho_{1}(t)\ast q_{\epsilon}\right)\left(\rho_{2}(t)\ast q_{\epsilon}\right), \left(\nabla G_{1}(\cdot,s)-\nabla G_{2}(\cdot,s)\right)\nabla H_{12}(\cdot,s)\rangle\dd s
    \right|\leq \delta \,.
\end{align}
Finally, letting $\hat{\epsilon}$ tend to $0$, we observe that $\bm{\rho}$ solves equation \eqref{HD-equations} in the sense of distributions.

\bibliography{reference2}
\bibliographystyle{unsrt}
\appendix
\newpage
\textbf{\LARGE{Appendix}}
\section{Proof of the superexponential estimate}\label{proof of superexponential estimate}
The objective of this appendix is to prove the superexponential estimate presented in Theorem \ref{superexponential estimate}. 
We follow here the road of the original paper \cite{KOV}, i.e., reducing the problem to one and two blocks estimates which then boil down to a uniform equivalence of ensembles. 
For the convenience of the reader and self-consistency of the paper, we nevertheless prefer to provide full details.
\subsection{Equivalence of ensembles}
In the following we denote by $\nu_{N}^{\alpha_{1},\alpha_{2}}$ the measure 
\begin{align}
    \nu^{N}_{\alpha_{1},\alpha_{2}}=\bigotimes_{x=1}^{N}\nu^{N,x}_{\alpha_{1},\alpha_{2}}\qquad \text{where}\qquad \nu^{N,x}_{\alpha_{1},\alpha_{2}}\sim \text{Multinomial}(1,1-\alpha_{1}-\alpha_{2},\alpha_{1},\alpha_{2})\,.
\end{align}
\begin{lemma}\label{lemma-condition}
    Given $k_{1},k_{2}\in \mathbb{N}_{0}$ such that $k_{1}+k_{2}\leq N$, the distribution $\nu_{\alpha_{1},\alpha_{2}}^{N}$ conditioned to the event $$\Omega_{k_1,k_2} := \left\{\bm \eta\ \Big|\ \sum_{x=1}^{N}\eta_{1}^{x}=|\eta_{1}|=k_{1},\,\sum_{x=1}^{N}\eta_{2}^{x}=|\eta_{2}|=k_{2}\right\}$$ is distributed as a uniform distribution of $k_{1},k_{2}$ particles of colours $1$ and $2$ respectively, over $N$ available sites. That is 
    \begin{align}
        \nu_{\alpha_{1},\alpha_{2}}^{N}\left(\bm{\eta}\bigg| |\eta_{1}|=k_{1},|\eta_{2}|=k_{2}\right)=
        \begin{cases}
            \binom{N}{k_{1},k_{2}}^{-1} \ \ \ \ &\text{if $\eta \in \Omega_{k_1,k_2}$}\\
            0&\text{else}
        \end{cases}
    \end{align}
\end{lemma}
\begin{proof}
We denote by $\Omega_{k_{1},k_{2}}$ the subspace of $\Omega$ where there are $k_{1},k_{2}$ colours $1$ and $2$ respectively. Then, we have that 
\begin{align}
    \nu_{\alpha_{1},\alpha_{2}}^{N}\left(\bm{\eta}\bigg| |\eta_{1}|=k_{1},|\eta_{2}|=k_{2}\right)=&\frac{\nu_{\alpha_{1},\alpha_{2}}^{N}\left(\bm{\eta}, |\eta_{1}|=k_{1},|\eta_{2}|=k_{2}\right)}{\nu_{\alpha_{1},\alpha_{2}}^{N}\left(\eta_{1}|=k_{1},|\eta_{2}|=k_{2}\right)}\nonumber
    \\=&
  \begin{cases}
      \frac{\alpha_{1}^{k}\alpha_{2}^{k}(1-\alpha_{1}-\alpha_{2})^{N-k_{1}-k_{2}}}{\sum_{\bm{\xi}\in \Omega_{k_{1},k_{2}}}\nu_{\alpha_{1},\alpha_{2}}^{N}(\bm{\xi})}\qquad &\text{if}\qquad |\eta_{1}|=k_{1}\quad \text{and}\quad |\eta_{2}|=k_{2}\,,\\
      0\qquad &\text{if}\qquad |\eta_{1}|\neq k_{1}\quad \text{or}\quad |\eta_{2}|\neq k_{2}\,.
  \end{cases}
\end{align}
By direct computations we write 
\begin{align}
\frac{\alpha_{1}^{k}\alpha_{2}^{k}(1-\alpha_{1}-\alpha_{2})^{N-k_{1}-k_{2}}}{\sum_{\bm{\xi}\in \Omega_{k_{1},k_{2}}}\nu_{\alpha_{1},\alpha_{2}}^{N}(\bm{\xi})}=\frac{\alpha_{1}^{k_{1}}\alpha_{2}^{k_{2}}(1-\alpha_{1}-\alpha_{2})^{N-k_{1}-k_{2}}}{\alpha_{1}^{k_{1}}\alpha_{2}^{k_{2}}(1-\alpha_{1}-\alpha_{2})^{N-k_{1}-k_{2}}}    \frac{1}{|\Omega_{k_{1},k_{2}}|}=\frac{1}{|\Omega_{k_{1},k_{2}}|}=\frac{1}{\binom{N}{k_{1},k_{2}}}\,.
\end{align}
\end{proof}
\begin{lemma}
    The measure $\nu_{\alpha_{1},\alpha_{2}}^{N}$ can be written as a convex combination of uniform measures, namely 
    \begin{equation}
        \nu_{\alpha_{1},\alpha_{2}}^{N}(\bm{\eta})=\sum_{k_{1}=0}^{N}\sum_{k_{2}=0}^{N}\mathbbm{1}_{\{k_{1}+k_{2}\leq N\}}\nu_{\alpha_{1},\alpha_{2}}^{N}\left(\bm{\eta}\bigg| |\eta_{1}|=k_{1},|\eta_{2}|=k_{2}\right)\nu_{\alpha_{1},\alpha_{2}}^{N}\left( |\eta_{1}|=k_{1},|\eta_{2}|=k_{2}\right)\,.
    \end{equation}
\end{lemma}
\begin{proof} The proof trivially follows from Lemma \ref{lemma-condition} and from the fact that 
\begin{equation}
    \sum_{k_{1}=0}^{N}\sum_{k_{2}=0}^{N}\mathbbm{1}_{\{k_{1}+k_{2}\leq N\}}\nu_{\alpha_{1},\alpha_{2}}^{N}\left( |\eta_{1}|=k_{1},|\eta_{2}|=k_{2}\right)=1
\end{equation}
with 
\begin{equation}
    \nu_{\alpha_{1},\alpha_{2}}^{N}\left( |\eta_{1}|=k_{1},|\eta_{2}|=k_{2}\right)\geq 0\qquad \forall
    k_{1},k_{2}\,.
\end{equation}
\end{proof}
For the sake of simplicity we denote by $\mu_{N,k_{1},k_{2}}$ the following distribution:
\begin{equation}
    \mu_{N,k_{1},k_{2}}(\bm{\eta})=\nu_{\alpha_{1},\alpha_{2}}^{N}\left(\bm{\eta}\bigg| |\eta_{1}|=k_{1},|\eta_{2}|=k_{2}\right)\,.
\end{equation} 
\begin{lemma}\label{Lemma-EE}[Equivalence of ensembles]
    Let $\ell \leq N$ and consider the subset of sites $\Lambda_{\ell}=\{0,1,\ldots,\ell\}\subset \mathbb{T}_{N}$. We introduce the configuration $\bm{\zeta}=(\zeta^{0},\ldots,\zeta^{\ell})$ over $\Lambda_{\ell}$, where $\zeta^{x}=(\zeta_{0}^{x},\zeta_{1}^{x},\zeta_{2}^{x})$  with $\zeta_{1}^{x},\zeta_{2}^{x}\in \{0,1\}$, satisfying the (exclusion) constraint  $\zeta_{0}^{x}=1-\zeta_{1}^{x}-\zeta_{2}^{x}$. We further denote by $m_{1}=|\zeta_{1}|$ and  $m_{2}=|\zeta_{2}|$, the number of particles of species $1$ and of species $2$ present in the subset $\Lambda_{\ell}$ respectively. Then,
    \begin{align}
        \lim_{N\to \infty,\; k_{1}/N\to \alpha_{1},\;k_{2}/N\to \alpha_{2}}\mu_{N,k_{1},k_{2}}\left(\left\{\bm{\eta}\;:\;\eta^{0}=\zeta^{0},\ldots, \eta^{\ell}=\zeta^{\ell}\right\}\right)=\alpha_{1}^{m_{1}}\alpha_{2}^{m_{2}}(1-\alpha_{1}-\alpha_{2})^{\ell-m_{1}-m_{2}}\,.
    \end{align}
\end{lemma}
\begin{proof} By direct computations we write
\begin{align}
    & \mu_{N,k_{1},k_{2}}\left(\left\{\bm{\eta}\;:\;\eta^{0}=\zeta^{0},\ldots \eta^{\ell}=\zeta^{\ell}\right\}\right)=\frac{\binom{N-\ell}{k_{1}-m_{1},k_{2}-m_{2}}}{\binom{N}{k_{1},k_{2}}}\nonumber 
     \\=&
     \frac{(N-\ell)!(N-k_{1}-k_{2})!k_{1}!k_{2}!}{N!(k_{1}-m_{1})!(k_{2}-m_{2})!(N-k_{1}-k_{2}-\ell+m_{1}+m_{2})!}\nonumber
     \\=&
     \frac{k_{1}(k_{1}-1)\cdots(k_{1}-m_{1}+1)}{N^{m_{1}}}\;\frac{k_{2}(k_{2}-1)\cdots (k_{2}-m_{2}+1)}{N^{m_{2}}}\nonumber\\ \times &\frac{(N-k_{1}-k_{2})(N-k_{1}-k_{2}-1)\cdots (N-k_{1}-k_{2}-\ell+m_{1}+m_{2}+1)}{N^{\ell-m_{1}-m_{2}}}\frac{N^{\ell}}{N(N-1)\cdots (N-\ell+1)},
\end{align}
where in the last equality we have used the properties of the factorials and we have multiplied by $\frac{N^{\ell}}{N^{\ell}}$. By taking the limit we have that
\begin{align}
    \lim_{N\to \infty,\; k_{1}/N\to \alpha_{1},\;k_{2}/N\to \alpha_{2}}\mu_{N,k_{1},k_{2}}\left(\left\{\bm{\eta}\;:\;\eta^{0}=\zeta^{0},\ldots \eta^{\ell}=\zeta^{\ell}\right\}\right)=\alpha_{1}^{m_{1}}\alpha_{2}^{m_{2}}(1-\alpha_{1}-\alpha_{2})^{\ell-k_{1}-k_{2}} \,.
\end{align}
\end{proof}
We call a function $\phi:\Omega \to \mathbb{R}$, with $\Omega$ defined as in \eqref{state space}, local if $\phi(\bm\eta)$ depends only on $\eta_1,...,\eta_\ell$ for some fixed $\ell$ not dependent on $N$
\begin{corollary}\label{corollary-EqEns}
    For every local function $\phi$, we have that
    \begin{align}\label{sup phi}
        \uplim_{N\to\infty}\sup_{0\leq k_{1},k_{2}\leq N\;:\; k_{1}+k_{2}\leq N}\left|\mathbb{E}_{\mu_{N,k_{1},k_{2}}}[\phi]-\mathbb{E}_{\nu^N_{\frac{k_{1}}{N},\frac{k_{2}}{N}}}[\phi]\right|=0
    \end{align}
\end{corollary}
\begin{proof}  For every finite $N$, the supremum over $k_{1}$ and $k_{2}$ is reached. Denote by $k_{1}^{*}(N)$ and $k_{2}^{*}(N)$ a value of $k$'s where the supremum is attained. Since 
\begin{equation}
    0\leq \frac{k_{1}^{*}(N)}{N}\leq 1\,,
\end{equation}
there exists a convergent subsequence of $k_{1}^{*}(N)/N$.
By consequence, as $\phi$ is local and by Lemma \ref{Lemma-EE}, we have 
\begin{align}
    \uplim_{i\to \infty} \left|\mathbb{E}_{\mu_{N,k_{1}^{*}(N_{i}),k_{2}^{*}(N_{i})}}[\phi]-\mathbb{E}_{\nu^N_{\frac{k_{1}^{*}(N_{i})}{N_{i}},\frac{k_{2}^{*}(N_{i})}{N_{i}}}}[\phi]\right|=0.
\end{align}
This holds for every possible converging subsequence of $\frac{k_1^*(N)}{N}$ and hence the statement follows. 
\end{proof}
\subsection{One and two blocks estimates}
In this section, our goal is to show that proving Theorem \ref{superexponential estimate} can be reduced to establishing two key lemmas, referred to as the one block and two blocks estimates, respectively. We hereby follow verbatim the steps of the proof of Theorem 2.1 of \cite{KOV}, with necessary adaptations to cover the multispecies case. The crucial aspect of this approach lies in the application of the Feynman-Kac formula (c.f. \cite[Proposition A.7.1]{KipnisLandim}).\\
We focus on the quantity $\mathbb{P}_{N}^{1/3}\left(\frac{1}{N}\int_{0}^{t}V_{N,\epsilon}(\bm{\eta}(s))\dd s\geq \delta\right)$ and we apply the exponential Chebyshev inequality, obtaining
\begin{align}\label{first ineq}
    \mathbb{P}_{N}^{1/3}\left(\frac{1}{N}\int_{0}^{t}V_{N,\epsilon}(\bm{\eta}(s))\dd s\geq \delta \right)\leq e^{-\delta N a }\mathbb{E}_{N}^{1/3}\left[\exp{\left(a\int_{0}^{t}V_{N,\epsilon}(\bm{\eta}(s))\dd s\right)}\right],
\end{align}
where we have denoted by $\mathbb{E}_{N}^{1/3}$ the expectation with respect to $\mathbb{P}_{N}^{1/3}$. 
To estimate the right hand side of \eqref{first ineq} we define the operator
\begin{equation}
    \mathcal{K}=\mathcal{L}+aV\,,
\end{equation}
i.e.,
\begin{align}
    \mathcal{K}f(\bm{\eta})=\mathcal{L}f(\bm{\eta})+aV(\bm{\eta})f(\bm{\eta}).
\end{align}
Using Feynman-Kac formula (see \cite[Proposition A.7.1, Lemma A.7.2]{KipnisLandim})
\begin{align}
    \mathbb{E}_{N}^{1/3}\left[\exp{\left(a\int_{0}^{t}V_{N,\epsilon}(\bm{\eta}(s))\dd s\right)}\right] 
    &=\left \langle 1, e^{t\mathcal{K}}1\right \rangle_{L^2(\nu^N_{1/3,1/3})}\nonumber \\
    &\leq \exp{\left(t\lambda_{max}(\mathcal{K})\right)},
\end{align}
where $\lambda_{max}(\mathcal{K})$ is the largest eigenvalue of the operator $\mathcal{K}$. 
It follows that
\begin{align}
    e^{-\delta N a }\mathbb{E}_{N}^{1/3}\left[\exp{\left(a\int_{0}^{t}V_{N,\epsilon}(\bm{\eta}(s))\dd s\right)}\right]
    &\leq \exp{\left(N\left(\frac{t}{N}\lambda_{max}(\mathcal{K})\right)-\delta a\right)}\,.
\end{align}
Therefore, to prove the superexponential estimate, it is enough to show that 
\begin{align}
    \uplim_{\epsilon\to 0}\;\uplim_{N\to\infty}\; \frac{1}{N}\lambda_{max}(\mathcal{K})=0\,.
\end{align}
For the largest eigenvalue we have the variational representation
\begin{align}
    \lambda_{max}(\mathcal{K})= \sup_{f_{N}}\left\{a\langle V_{N,\epsilon}(\bm{\eta}),f_{N}\rangle_{\nu_{1/3,1/3}^{N}}-N^{2}D_{N}(f_{N})\right\}\,,
\end{align}
where the supremum is taken over probability densities $f_N$, i.e., $f_N\geq 0$ and $\sum_{\bm{\eta}\in \Omega}f_{N}(\bm{\eta})3^{-N}=1$. Furthermore, $D_N$ is the so-called \emph{Dirichlet form} associated with the generator $\mathcal{L}$, and is given by 
\begin{align}
    D_{N}(f_{N})&=\frac{1}{2}\sum_{\bm{\eta},\bm{\xi}\in \Omega}\nu_{1/3,1/3}^{N}(\bm{\eta})\mathcal{L}(\bm{\eta},\bm{\xi})\left(\sqrt{f_{N}(\bm{\xi})}-\sqrt{f_{N}(\bm{\eta})}\right)^{2}\nonumber\\
    &=
    \frac{3^{-N}}{2}\sum_{\bm{\eta}\in \Omega}\sum_{x=1}^{N}\sum_{\alpha,\beta=0}^{2}\eta_{\alpha}^{x}\eta_{\beta}^{x+1}\left(\sqrt{f_{N}(\bm{\eta}_{\alpha,\beta}^{x,x+1})}-\sqrt{f_{N}(\bm{\eta}})\right)^{2},
\end{align}
and where
\begin{align}
    \langle V_{N,\epsilon}(\bm{\eta}),f_{N}\rangle_{\nu_{1/3,1/3}^{N}}=\sum_{\bm{\eta}\in \Omega}V_{N,\epsilon}(\bm{\eta})f_{N}(\bm{\eta})3^{-N}\,.
\end{align}
 Since $\phi(\bm{\eta})$ is bounded, there exists a positive constant $C$ such that 
\begin{equation}
    \langle V_{N,\epsilon}(\bm{\eta}),f_{N}\rangle_{\nu_{1/3,1/3}^{N}}\leq C N\,.
\end{equation}
As a result, we restrict the supremum to the set of densities $f_{N}$ that satisfy $D_{N}(f_{N})\leq C/N$. Furthermore, we consider only the densities $f_{N}$ that are translation invariant (since $D_{N}(\cdot)$ is convex, for details see Appendix 10 of \cite{KipnisLandim}). Consequently, we obtain the estimate
\begin{align}
    &\sup_{f_{N}\,:\, D_{N}(f_{N})\leq C/N}\left\{a\sum_{\bm{\eta}\in \Omega} V_{N,\epsilon}(\bm{\eta})f_{N}(\bm{\eta})3^{-N}-N^{2}D_{N}(f_{N})\right\}\nonumber
    \\&
    \leq  \sup_{f_{N}\,:\, D_{N}(f_{N})\leq C/N}\left\{a\sum_{\bm{\eta}\in \Omega} V_{N,\epsilon}(\bm{\eta})f_{N}(\bm{\eta})3^{-N}\right\}\,.
\end{align}
This implies that is enough to show that 
\begin{align}\label{intermedaite-goal-SEE}
    \uplim_{\epsilon\to 0}\;\uplim_{N\to\infty} \sup_{f_{N}\,:\, D_{N}(f_{N})\leq C/N}\frac{1}{N}\left\{\sum_{\bm{\eta}\in \Omega} V_{N,\epsilon}(\bm{\eta})f_{N}(\bm{\eta})3^{-N}\right\}=0\,.
\end{align}
Writing out the definition of $V_{N,\epsilon}(\bm{\eta})$ given in definition \eqref{definition-V},  we obtain, using translation invariance of $f$,
\begin{align}\label{explicit-EigenV}
     \uplim_{\epsilon\to 0}\;\uplim_{N\to\infty} \sup_{f_{N}\,:\, D_{N}(f_{N})\leq C/N}&\left\{\sum_{\bm{\eta}\in \Omega}\left|\frac{1}{2N \epsilon+1}\sum_{|x|\leq \epsilon N}\tau_{x}\phi(\bm{\eta})\right.\right.\nonumber
     \\&\left.\left.-\widetilde{\phi}\left(\frac{1}{2N \epsilon+1}\sum_{|x|\leq \epsilon N}\eta_{1}^{x},\frac{1}{2N \epsilon+1}\sum_{|x|\leq \epsilon N}\eta_{2}^{x}\right)\right|f_{N}(\bm{\eta})3^{-N}\right\}=0\,.
\end{align}
For any fixed $y\in \mathbb{T}_{N}$, we consider a neighborhood of discrete points $\{y-k,y-k+1,\ldots,y+k-1,y+k\}$. Within this neighborhood, we have the following approximation
\begin{align}\label{trick}
    \frac{1}{N} \sum_{y} \tau_{y}\phi(\bm{\eta})=\frac{1}{N} \sum_{y} \frac{1}{2k+1}\sum_{|z-y|\leq k}\tau_{z}\phi(\bm{\eta})+\mathcal{O}\left(\frac{k}{N}\right).
\end{align}
 Next we add and subtract the quantity $\widetilde{\phi}\left(\frac{1}{2k+1}\sum_{|z-x|\leq k}\eta_{1}^{z},\frac{1}{2k+1}\sum_{|z-x|\leq k}\eta_{2}^{z}\right)$ inside the absolute value of equation \eqref{explicit-EigenV}, obtaining
\begin{align}\label{usefull}
    &\left|\frac{1}{2\epsilon N+1}\sum_{|x|\leq \epsilon N}\tau_{x}\phi(\bm{\eta})-\widetilde{\phi}\left(\frac{1}{2 \epsilon N+1}\sum_{|y|\leq \epsilon N}\eta_{1}^{y},\frac{1}{2\epsilon N+1}\sum_{|y|\leq \epsilon N}\eta_{2}^{y}\right)\right|\nonumber\\
    \leq &
    \frac{1}{2\epsilon N+1}\sum_{|x|\leq \epsilon N}\left|\frac{1}{2k+1}\sum_{|z-x|\leq k}\tau_{z}\phi(\bm{\eta})-\widetilde{\phi}\left(\frac{1}{2k+1}\sum_{|x-z|\leq k}\eta_{1}^{z},\frac{1}{2k+1}\sum_{|x-z|\leq k}\eta_{2}^{z}\right)\right|\nonumber\\
    +&
    \frac{1}{2\epsilon N+1}\sum_{|x|\leq \epsilon N}\left|\widetilde{\phi}\left(\frac{1}{2k+1}\sum_{|x-z|\leq k}\eta_{1}^{z},\frac{1}{2k+1}\sum_{|x-z|\leq k}\eta_{2}^{z}\right)\right.\nonumber
    \\&\qquad\qquad\qquad\quad-\left.\widetilde{\phi}\left(\frac{1}{2 \epsilon N+1}\sum_{|y|\leq \epsilon N}\eta_{1}^{y},\frac{1}{2\epsilon N+1}\sum_{|y|\leq \epsilon N}\eta_{2}^{y}\right)\right|+\mathcal{O}\left(\frac{k}{N}\right)\,.
\end{align}
We consider the second addend in the right-hand-side of \eqref{usefull}. By exploiting the multi-variable mean-value theorem and \eqref{trick} we have that
\begin{align}
    &\frac{1}{(2\epsilon N+1)}\sum_{|x|\leq \epsilon N}\left|\widetilde{\phi}\left(\frac{1}{2k+1}\sum_{|x-z|\leq k}\eta_{1}^{z},\frac{1}{2k+1}\sum_{|x-z|\leq k}\eta_{2}^{z}\right)\right.\nonumber
    \\&\qquad\qquad\qquad\quad\;-\left.\widetilde{\phi}\left(\frac{1}{2 \epsilon N+1}\sum_{|y|\leq \epsilon N}\eta_{1}^{y},\frac{1}{2\epsilon N+1}\sum_{|y|\leq \epsilon N}\eta_{2}^{y}\right)\right|\nonumber
    \\\leq &
    \frac{\lVert \nabla \widetilde{\phi}\rVert_{\infty}}{(2\epsilon N+1)}\sum_{|x|\leq \epsilon N}\left\lVert\left(\frac{1}{2k+1}\sum_{|x-z|\leq k}\eta_{1}^{z},\frac{1}{2k+1}\sum_{|x-z|\leq k}\eta_{2}^{z}\right)\right.\nonumber
    \\&\qquad\qquad\qquad\quad-\left.\left(\frac{1}{2\epsilon N+1}\sum_{|y|\leq \epsilon N}\eta_{1}^{y},\frac{1}{2\epsilon N+1}\sum_{|y|\leq \epsilon N}\eta_{2}^{y}\right)\right\rVert_{2}  \nonumber\\
    \leq &
    \frac{\lVert \nabla \widetilde{\phi}\rVert_{\infty}}{(2\epsilon N+1)^{2}}\sum_{|x|\leq \epsilon N}\sum_{|y|\leq \epsilon N}\left\lVert\left(\frac{1}{2k+1}\sum_{|x-z|\leq k}\eta_{1}^{z},\frac{1}{2k+1}\sum_{|x-z|\leq k}\eta_{2}^{z}\right)\right.\nonumber
    \\&\qquad\qquad\qquad\quad-\left.\left(\frac{1}{2k+1}\sum_{|y-z|\leq k}\eta_{1}^{z},\frac{1}{2k+1}\sum_{|y-z|\leq k}\eta_{2}^{z}\right)\right\rVert_{2}+\mathcal{O}\left(\frac{k}{N}\right)\,.
\end{align}
It follows that
\begin{align}
    &\left|\frac{1}{2\epsilon N+1}\sum_{|x|\leq \epsilon N}\tau_{x}\phi(\bm{\eta})-\widetilde{\phi}\left(\frac{1}{2 \epsilon N+1}\sum_{|y|\leq \epsilon N}\eta_{1}^{y},\frac{1}{2\epsilon N+1}\sum_{|y|\leq \epsilon N}\eta_{2}^{y}\right)\right|\nonumber\\
    \leq &
    \frac{1}{2\epsilon N+1}\sum_{|x|\leq \epsilon N}\left|\frac{1}{2k+1}\sum_{|z-x|\leq k}\tau_{z}\phi(\bm{\eta})-\widetilde{\phi}\left(\frac{1}{2k+1}\sum_{|x-z|\leq k}\eta_{1}^{z},\frac{1}{2k+1}\sum_{|x-z|\leq k}\eta_{2}^{z}\right)\right|\nonumber\\
    +&
    \frac{\lVert \nabla \widetilde{\phi}\rVert_{\infty}}{(2\epsilon N+1)^{2}}\sum_{|x|\leq \epsilon N}\sum_{|y|\leq \epsilon N}\left\lVert\left(\frac{1}{2k+1}\sum_{|x-z|\leq k}\eta_{1}^{z},\frac{1}{2k+1}\sum_{|x-z|\leq k}\eta_{2}^{z}\right)\right.\nonumber
    \\&\qquad\qquad\qquad\qquad\qquad\quad-\left.\left(\frac{1}{2k+1}\sum_{|y-z|\leq k}\eta_{1}^{z},\frac{1}{2k+1}\sum_{|y-z|\leq k}\eta_{2}^{z}\right)\right\rVert_{2}+\mathcal{O}\left(\frac{k}{N}\right)\,.
\end{align}
Arrived at this point, in order to obtain \eqref{intermedaite-goal-SEE} it is sufficient to prove the following two lemmas:
\begin{lemma}[One block estimate]\label{lemma-OBE}
    For all $c>0$
    \begin{align}
        &\uplim_{k\to \infty}\; \uplim_{N\to\infty} \sup_{f_{N}\,:\, D(f_{N})\leq c/N}\nonumber
        \\&
        \sum_{\bm{\eta}\in \Omega}\left|\frac{1}{2k+1}\sum_{|z|\leq k}\tau_{z}\phi(\bm{\eta})-\widetilde{\phi}\left(\frac{1}{2k+1}\sum_{|z|\leq k}\eta_{1}^{z},\frac{1}{2k+1}\sum_{|z|\leq k}\eta_{2}^{z}\right)\right|f_{N}(\bm{\eta})3^{-N}=0\,.
    \end{align}
\end{lemma}
\begin{lemma}[Two blocks estimate]\label{lemma-TBE}
    For all $c> 0$
    \begin{align}
        &\uplim_{k\to \infty}\;\uplim_{\epsilon\to 0}\;\uplim_{N\to\infty} \sup_{|r|\leq 2\epsilon N+1}\sup_{f_{N}\,:\, D(f_{N})\leq c/N}\nonumber
        \\&
        \sum_{\bm{\eta}\in \Omega}\left\lVert\left(\frac{1}{2k+1}\sum_{|z|\leq k}\eta_{1}^{z},\frac{1}{2k+1}\sum_{|z|\leq k}\eta_{2}^{z}\right)-\left(\frac{1}{2k+1}\sum_{|z+r|\leq k}\eta_{1}^{z},\frac{1}{2k+1}\sum_{|z+r|\leq k}\eta_{2}^{z}\right)\right\rVert_{2}f_{N}(\bm{\eta})3^{-N}=0\,.
    \end{align}
\end{lemma}
\subsection{Proof of the one block estimate}
Fix $k\in \mathbb{N}$ such that $k\leq N$, and consider the set  $\{x\in \mathbb{T}_{N}\,:\,|x|\leq k\}$. We introduce the subspace $\Omega_{2k+1}\subset \Omega$, which represents the state space restricted to these $2k+1$ sites. Then, for any function $g:\Omega_{2k+1}\to \mathbb{R}$, we define the "restricted" Dirichlet form as follows:
\begin{align}
    D_{2k+1}^{*}(g)=\frac{1}{2}\sum_{\bm{\eta}\in \Omega_{2k+1}}3^{-(2k+1)}\sum_{x=-k}^{k-1}\sum_{\alpha,\beta=0}^{2}\eta_{\alpha}^{x}\eta_{\beta}^{x+1}\left(\sqrt{g(\bm{\eta}_{\alpha,\beta}^{x,x+1})}-\sqrt{g(\bm{\eta})}\right)^{2}\,.
\end{align}
Next, we define the marginal of the density $f_{N}$ over $\Omega_{2k+1}$ as 
\begin{align}
    f_{N}^{k}(\bm{\eta})=3^{-N+2k+1}\sum_{\eta^{x}\;:\;|x|>k}f_{N}(\bm{\eta})\,.
\end{align}
Using the following inequality
\begin{align}\label{sumSquare-squareSum}
    \left(\sqrt{\smash[b]{\sum_{j}a_{j}}\vphantom{\sum a_j}}-\sqrt{\smash[b]{\sum_{j}b_{j}}\vphantom{\sum a_j}} \right)^{2}\leq \sum_{j}\left(\sqrt{a_{j}}-\sqrt{b_{j}}\right)^{2},
\end{align}
we have that
\begin{align}
    D_{2k+1}^{*}(f_{N}^{k})&=\frac{1}{2}\sum_{\bm{\eta}\in \Omega_{2k+1}}3^{-N}\sum_{x=-k}^{k-1}\sum_{\alpha,\beta=0}^{2}\eta_{\alpha}^{x}\eta_{\beta}^{x+1}\left(\sqrt{\sum_{(\eta^x)_{|x|>k}}f_{N}(\bm{\eta}_{\alpha,\beta}^{x,x+1})}-\sqrt{\sum_{(\eta^x)_{|x|>k}}f_{N}(\bm{\eta})}\right)^{2}\nonumber
    \\&\leq
    \frac{1}{2}\sum_{\bm{\eta}\in \Omega}3^{-N}\sum_{x=-k}^{k-1}\sum_{\alpha,\beta=0}^{2}\eta_{\alpha}^{x}\eta_{\beta}^{x+1}\left(\sqrt{f_{N}(\bm{\eta}_{\alpha,\beta}^{x,x+1})}-\sqrt{f_{N}(\bm{\eta})}\right)^{2}\nonumber
    \\ &=
    \frac{1}{2}\sum_{x=-k}^{k-1}\sum_{\bm{\eta}\in \Omega}3^{-N}\sum_{\alpha,\beta=0}^{2}\eta_{\alpha}^{0}\eta_{\beta}^{1}\left(\sqrt{f_{N}(\bm{\eta}_{\alpha,\beta}^{0,1})}-\sqrt{f_{N}(\bm{\eta})}\right)^{2}\nonumber
    \\ &=
    \frac{2k}{N}D(f_{N})\,.
\end{align}
Here, in the up to last equality we have used the translation invariance. All in all, we obtain the upper bound
\begin{equation}
    D_{2k+1}^{*}(f_{N}^{k})\leq \frac{2k}{N}D(f_{N})\,.
\end{equation}\\
As a consequence 
\begin{align}
    &\sup_{f_{N}\,:\, D(f_{N})\leq c/N }\sum_{\bm{\eta}\in \Omega}\left|\frac{1}{2k+1}\sum_{|z|\leq k}\tau_{z}\phi(\bm{\eta})-\widetilde{\phi}\left(\frac{1}{2k+1}\sum_{|z|\leq k}\eta_{1}^{z},\frac{1}{2k+1}\sum_{|z|\leq k}\eta_{2}^{z}\right)\right|f_{N}(\bm{\eta})3^{-N}\nonumber
    \\ \leq &
    \sup_{g_{k}\,:\, D_{2k+1}^{*}(g_{k})\leq (2ck)/N^{2}}\sum_{\bm{\eta}\in \Omega}\left|\frac{1}{2k+1}\sum_{|z|\leq k}\tau_{z}\phi(\bm{\eta})\right.\nonumber
    \\&\qquad \qquad\qquad\qquad\qquad \; -\left.\widetilde{\phi}\left(\frac{1}{2k+1}\sum_{|z|\leq k}\eta_{1}^{z},\frac{1}{2k+1}\sum_{|z|\leq k}\eta_{2}^{z}\right)\right|g_{k}(\bm{\eta})3^{-2k-1}+\mathcal{O}\left(\frac{k}{N}\right)\,.
\end{align}
Taking the limsup as $N\to \infty$ and using the compactness of the level sets of the Dirichlet form (for details see Appendix 10 of \cite{KipnisLandim}), we have
\begin{align}\label{unifomr-step}
      &\uplim_{N\to\infty}\sup_{g_{k}\,:\, D_{2k+1}^{*}(g_{k})\leq (2ck)/N^{2}}\sum_{\bm{\eta}\in \Omega}\left|\frac{1}{2k+1}\sum_{|z|\leq k}\tau_{z}\phi(\bm{\eta})-\widetilde{\phi}\left(\frac{1}{2k+1}\sum_{|z|\leq k}\eta_{1}^{z},\frac{1}{2k+1}\sum_{|z|\leq k}\eta_{2}^{z}\right)\right|g_{k}(\bm{\eta})3^{-2k-1}\nonumber
    \\ \leq &
    \sup_{g_{k}\,:\, D_{2k+1}^{*}(g_{k})=0}\sum_{\bm{\eta}\in \Omega}\left|\frac{1}{2k+1}\sum_{|z|\leq k}\tau_{z}\phi(\bm{\eta})-\widetilde{\phi}\left(\frac{1}{2k+1}\sum_{|z|\leq k}\eta_{1}^{z},\frac{1}{2k+1}\sum_{|z|\leq k}\eta_{2}^{z}\right)\right|g_{k}(\bm{\eta})3^{-2k-1}\,.
\end{align}
The set of probability distribution with density $g_{k}$ such that $D_{2k+1}^{*}(g_{k})=0$ is the set of uniform distributions over $\Omega_{2k+1}$ with fixed number of particles $k_{1},k_{2}$ of species $1$ and $2$ respectively. Therefore, taking the supremum in equation \eqref{unifomr-step} is equivalent to taking the supremum over all configurations $\bm{\eta}$ in the space $\Omega_{2k+1}$ with fixed number of particles $k_{1}$ and $k_{2}$ of the two species. As a consequence, by taking the limsup for $k\to\infty$, we have that  
\begin{align}
    &\uplim_{k\to \infty}\sup_{g_{k}\,:\, D_{2k+1}^{*}(g_{k})=0}\sum_{\bm{\eta}\in \Omega}\left|\frac{1}{2k+1}\sum_{|z|\leq k}\tau_{z}\phi(\bm{\eta})-\widetilde{\phi}\left(\frac{1}{2k+1}\sum_{|z|\leq k}\eta_{1}^{z},\frac{1}{2k+1}\sum_{|z|\leq k}\eta_{2}^{z}\right)\right|g_{k}(\bm{\eta})3^{-2k-1}\nonumber
    \\=&
    \uplim_{k\to \infty}\sup_{k_{1},k_{2}=0,\ldots,2k+1\;:\;k_{1}+k_{2}\leq 2k+1}\nonumber
    \\&
    \sum_{\bm{\eta}\in \Omega_{2k+1}\;:\; |\eta_{1}|=k_{1},|\eta_{2}|=k_{2}}\frac{\left|\frac{1}{2k+1}\sum_{|z|\leq k}\tau_{z}\phi(\bm{\eta})-\widetilde{\phi}\left(\frac{1}{2k+1}\sum_{|z|\leq k}\eta_{1}^{z},\frac{1}{2k+1}\sum_{|z|\leq k}\eta_{2}^{z}\right)\right|}{\binom{2k+1}{k_{1},k_{2}}}\nonumber
    \\=&
    \uplim_{k\to \infty}\sup_{k_{1},k_{2}=0,\ldots,k\;:\;k_{1}+k_{2}\leq 2k+1}\left|\mathbb{E}_{\mu_{k,k_{1},k_{2}}}[\phi] -\mathbb{E}_{\nu_{\frac{k_{1}}{2k+1},\frac{k_{2}}{2k+1}}^{k}}[\phi]\right|\nonumber
    \\=&0\,.
\end{align}
Here, in the last step, we used Corollary \ref{corollary-EqEns}. 
\subsection{Proof of the two blocks estimate}
In analogy to the approach used in the proof of Lemma \ref{lemma-OBE}, we now consider two blocks of size $2k+1$: the first centered around the microscopic point $0\in \mathbb{T}_{N}$ and the second centered around the microscopic point $r\in \mathbb{T}_{N}$. The centers of these two blocks are separated by a distance of at most $2\epsilon N+1$.  We denote by $\bm{\zeta},\bm{\xi}$ the configurations in the first and second block respectively, both belonging to the sub-space $\Omega_{2k+1}$. We consider an arbitrary function $g:\Omega_{2k+1}\times \Omega_{2k+1}\to \mathbb{R}$ and we define the following "restricted" Dirichlet-forms:
\begin{align}
    D_{k}^{1}(g)&=\frac{1}{2}\sum_{\bm{\zeta},\bm{\xi}\in \Omega_{2k+1}}3^{-4k-2}\sum_{x=-k}^{k-1}\sum_{\alpha,\beta=0}^{2}\zeta_{\alpha}^{x}\zeta_{\beta}^{x+1}\left(\sqrt{g(\bm{\zeta}_{\alpha,\beta}^{x,x+1},\bm{\xi})}-\sqrt{g(\bm{\zeta},\bm{\xi})}\right)^{2}\\
    D_{k}^{2}(g)&=\frac{1}{2}\sum_{\bm{\zeta},\bm{\xi}\in \Omega_{2k+1}}3^{-4k-2}\sum_{x=-k}^{k-1}\sum_{\alpha,\beta=0}^{2}\xi_{\alpha}^{x}\xi_{\beta}^{x+1}\left(\sqrt{g(\bm{\zeta},\bm{\xi}_{\alpha,\beta}^{x,x+1})}-\sqrt{g(\bm{\zeta},\bm{\xi})}\right)^{2}\\
    \Delta_{k}(g)&=\frac{1}{2}\sum_{\bm{\zeta},\bm{\xi}\in \Omega_{2k+1}}3^{-4k-2}\left(\sqrt{g(\bm{\zeta},\bm{\xi})^{0}}-\sqrt{g(\bm{\zeta},\bm{\xi})}\right)^{2}
\end{align}
where $(\bm{\zeta},\bm{\xi})^{0}$ indicates the configurations where the occupation variables at the center points of the two blocks have been exchanged. Intuitively, the first Dirichlet-form concerns the first block; the second Dirichlet-form the second block; the third Dirichlet-form takes into account the transfer of particles from one block to the other. We now introduce the marginal over the two blocks
\begin{align}\label{marginal-two-blocks}
    f_{N}^{r,k}(\bm{\eta})=3^{-N+4k+2}\sum_{\eta^{x}\;:\; |x|>k,\;|x-r|>k}f_{N}(\bm{\eta})\,.
\end{align}
Arguing as in the proof of Lemma \ref{lemma-OBE}, one can show the follwing estimates:
\begin{align}
    D_{k}^{1}(f_{N}^{r,k})&\leq \frac{2k}{N}D(f_{N})\\
    D_{k}^{2}(f_{N}^{r,k})&\leq \frac{2k}{N}D(f_{N})\,.
\end{align}
We now aim to find an upper bound for the Dirichlet-form $\Delta_{k}(\cdot)$ in terms of $\epsilon$. Obtaining the configuration $(\bm{\zeta},\bm{\xi})^{0}$ from the configuration $(\bm{\zeta},\bm{\xi})$ is equivalent to permuting the occupation variables  $\eta^{0}$ and $\eta^{r}$. We introduce the permutation operator $P_{x,y}$ between sites $x$ and $y$, defined as follows:
\begin{equation}
    P_{x,y}\bm{\eta}=\left(\eta^{0},\ldots,\eta^{x-1},\eta^{y},\eta^{x+1},\ldots,\eta^{y-1},\eta^{x},\eta^{y+1},\ldots,\eta^{N}\right)\,.
\end{equation}
By applying \eqref{sumSquare-squareSum} and by the definition of the marginal over the two blocks written in \eqref{marginal-two-blocks} we obtain
\begin{align}\label{delta frk}
    \Delta_{k}(f_{N}^{r,k})&=\frac{1}{2}\sum_{\bm{\zeta},\bm{\xi}\in \Omega_{2k+1}}3^{-4k-2}\left(\sqrt{f_{N}^{r,k}(\bm{\zeta},\bm{\xi})^{0}}-\sqrt{f_{N}^{r,k}(\bm{\zeta},\bm{\xi})}\right)^{2}\nonumber
    \\&\leq
    \frac{1}{2}\sum_{\bm{\eta}\in \Omega}3^{-N}\left(\sqrt{f_{N}(P_{0,r}\bm{\eta})}-\sqrt{f_{N}(\bm{\eta})}\right)^{2}\,.
\end{align}
This permutation operator satisfies the property\footnote{that can be proved by using the fact that $P_{i,j}=P_{j,i}$ and $P_{i,j}P_{j,k}=P_{j,k}P_{i,k}$.}
\begin{align}
  P_{1,2}P_{3,2}P_{2,1}=P_{1,2}P_{2,1}P_{1,3}=P_{1,3}  \,.
\end{align}
Therefore, we have that 
\begin{align}\label{permutation-firt-application}
    &\left(\sqrt{f_{N}(P_{0,r}\bm{\eta})}-\sqrt{f_{N}(\bm{\eta})}\right)^{2}\nonumber\\=&\left(\sqrt{f_{N}(P_{0,1}\bm{\eta})}-\sqrt{f_{N}(\bm{\eta})}+\sqrt{f_{N}(P_{0,1}P_{1,2}\bm{\eta})}-\sqrt{f_{N}(P_{0,1}\bm{\eta})}\right.\nonumber\\+&\left. \sqrt{f_{N}(P_{2,3}P_{1,2}P_{0,1}\bm{\eta})}-\sqrt{f_{N}(P_{1,2}P_{0,1}\bm{\eta})}+\ldots \right.\nonumber\\ +&\left.\sqrt{f_{N}(P_{1,0}\cdots P_{r-1,r-2}P_{r-1,r}\cdots P_{0,1}\bm{\eta})}-\sqrt{f_{N}(P_{2,1}\cdots P_{r-1,r-2}P_{r-1,r}\cdots P_{0,1}\bm{\eta})}\right)^{2}\nonumber\\
    \leq & (2r-1)\left\{\left(\sqrt{f_{N}(P_{0,1}\bm{\eta})}-\sqrt{f_{N}(\bm{\eta})}\right)^{2}+\left(\sqrt{f_{N}(P_{0,1}P_{1,2}\bm{\eta})}-\sqrt{f_{N}(P_{0,1}\bm{\eta})}\right)^{2}\right.\nonumber\\+&\left. \left(\sqrt{f_{N}(P_{2,3}P_{1,2}P_{0,1}\bm{\eta})}-\sqrt{f_{N}(P_{1,2}P_{0,1}\bm{\eta})}\right)^{2}+\ldots \right.\nonumber\\ +&\left.\left(\sqrt{f_{N}(P_{1,0}\cdots P_{r-1,r-2}P_{r-1,r}\cdots P_{0,1}\bm{\eta})}-\sqrt{f_{N}(P_{2,1}\cdots P_{r-1,r-2}P_{r-1,r}\cdots P_{0,1}\bm{\eta})}\right)^{2}\right\}\,.
\end{align}
Consequently, we find that 
\begin{align}\label{estimate-DELTA}
     &\frac{1}{2}\sum_{\bm{\eta}\in \Omega}3^{-N}\left(\sqrt{f_{N}(P_{0,r}\bm{\eta})}-\sqrt{f_{N}(\bm{\eta})}\right)^{2}\nonumber\\ &\leq 
     (2r-1)^{2}\frac{1}{2} \sum_{\bm{\eta}\in \Omega}3^{-N}\sum_{\alpha,\beta=0}^{2}\eta_{\alpha}^{0}\eta_{\beta}^{1}\left(\sqrt{f_{N}(\bm{\eta}_{\alpha,\beta}^{0,1})}-\sqrt{f_{N}(\bm{\eta})}\right)^{2},
\end{align}
where we used the  translation invariance of $f_N$. 
Therefore, using \eqref{delta frk}, it follows that
\begin{align}
    \Delta_{k}(f_{N}^{r,k})
    \leq \frac{(2r-1)^{2}}{N}D(f_{N})\,.
\end{align}
Finally,  for fixed $c>0$, $\epsilon>0$ and $N\in \mathbb{N}$, we define the set 
\begin{align}
    A_{N,\epsilon}:=\left\{g\;:\;D_{k}^{1}(g)\leq \frac{2c k}{N^{2}},\;D_{k}^{2}(g)\leq \frac{2c k}{N^{2}},\;\Delta_{k}(g)\leq \epsilon^{2}c\right\}\,.
\end{align}
Arguing as in the proof of Lemma \ref{lemma-OBE} it follows that  
\begin{equation}
   \left\{f_{N}\;:\; D(f_{N})\leq c/N,\right\}\cap\left\{ r\;:\; |r|\leq \epsilon N\right\}\subset\left\{g\;:\;g\in A_{N,\epsilon}\right\}\,.
\end{equation}
The above inclusion relation implies that 
\begin{align}
    &\sup_{|r|\leq \epsilon N}\sup_{f_{N}\;:\; D(f_{N})\leq c/N}\sum_{\bm{\eta}\in \Omega}\left\lVert\left(\frac{1}{2k+1}\sum_{|z|\leq k}\eta_{1}^{z},\frac{1}{2k+1}\sum_{|z|\leq k}\eta_{2}^{z}\right)\right.\nonumber
    \\&\qquad\qquad\qquad\qquad\qquad\quad-\left.\left(\frac{1}{2k+1}\sum_{|z+r|\leq k}\eta_{1}^{z},\frac{1}{2k+1}\sum_{|z+r|\leq k}\eta_{2}^{z}\right)\right\rVert_{2}f_{N}(\bm{\eta})3^{-N}\nonumber
    \\ \leq &
    \sup_{g\in A_{N,\epsilon}}\sum_{\bm{\zeta},\bm{\xi}\in \Omega_{2k+1}}\left\lVert\left(\frac{1}{2k+1}\sum_{x=- k}^{k}\zeta_{1}^{x},\frac{1}{2k+1}\sum_{x=-k}^{k}\zeta_{2}^{x}\right)\right.\nonumber
    \\&\qquad\qquad\qquad\qquad\qquad\quad-\left.\left(\frac{1}{2k+1}\sum_{x=- k}^{k}\xi_{1}^{x},\frac{1}{2k+1}\sum_{x=-k}^{k}\xi_{2}^{x}\right)\right\rVert_{2}g(\bm{\zeta},\bm{\xi})3^{-4k-2}\,.
    \end{align}
By taking the limsup for $N\to \infty$ and $\epsilon\to 0$, by exploiting the compactness of the level sets of the Dirichlet-forms and by using the fact that for all $\bm{x}\in \mathbb{R}^{2}$ it holds $\lVert \bm{x}\rVert_{1}\geq \lVert \bm{x}\rVert_{2}$ \footnote{Here we denoted by $\lVert \bm{x}\rVert_{1}:=|x_{1}|+|x_{2}|$ and $\lVert \bm{x}\rVert_{2}:=\sqrt{|x_{1}|^{2}+|x_{2}|^{2}}$} we obtain 
\begin{align}
    &\sup_{g\;:\;D_{k}^{1}(g)=0\,,D_{k}^{2}(g)=0,\,\Delta_{k}(g)=0}\sum_{\bm{\zeta},\bm{\xi}\in \Omega_{2k+1}}\left\lVert\left(\frac{1}{2k+1}\sum_{x=- k}^{k}\zeta_{1}^{x},\frac{1}{2k+1}\sum_{x=-k}^{k}\zeta_{2}^{x}\right)\right. \nonumber\\&\left.\qquad \qquad \qquad \qquad \qquad \qquad \qquad\qquad\; -\left(\frac{1}{2k+1}\sum_{x=- k}^{k}\xi_{1}^{x},\frac{1}{2k+1}\sum_{x=-k}^{k}\xi_{2}^{x}\right)\right\rVert_{2}g(\bm{\zeta},\bm{\xi})3^{-4k-2}\nonumber 
    \\ \leq &
     \sup_{g\;:\;D_{k}^{1}(g)=0\,,D_{k}^{2}(g)=0,\,\Delta_{k}(g)=0}\sum_{\bm{\zeta},\bm{\xi}\in \Omega_{2k+1}}\left|\left(\frac{1}{2k+1}\sum_{x=- k}^{k}\zeta_{1}^{x}\right)-\left(\frac{1}{2k+1}\sum_{x=- k}^{k}\xi_{1}^{x}\right)\right|g(\bm{\zeta},\bm{\xi})3^{-4k-2}\nonumber
     \\+&
      \sup_{g\;:\;D_{k}^{1}(g)=0\,,D_{k}^{2}(g)=0,\,\Delta_{k}(g)=0}\sum_{\bm{\zeta},\bm{\xi}\in \Omega_{2k+1}}\left|\left(\frac{1}{2k+1}\sum_{x=-k}^{k}\zeta_{2}^{x}\right)-\left(\frac{1}{2k+1}\sum_{x=-k}^{k}\xi_{2}^{x}\right)\right|g(\bm{\zeta},\bm{\xi})3^{-4k-2}\,.
\end{align}
The set of distributions that satisfy $D_{k}^{1}(g)=D_{k}^{2}(g)=\Delta_{k}(g)=0$ is the set of uniform distributions over $\Omega_{2k+1}\times \Omega_{2k+1}$, with fixed numbers $k_{1}$ and $k_{2}$ of particles of species $1$ and $2$ respectively. We choose the function $\phi^{1}: \Omega_{2k+1}\to \mathbb{R}$ as 
\begin{align}
    \phi^{1}(\bm{\zeta})=\zeta_{1}^{0}\qquad \text{and}\qquad  \phi^{1}(\bm{\xi})=\xi_{1}^{0}, 
\end{align}
and the function $\phi^{2}:\Omega_{2k+1}\to \mathbb{R}$ as 
\begin{align}
  \phi^{2}(\bm{\zeta})=\zeta_{2}^{0}\qquad \text{and}\qquad   \phi^{2}(\bm{\xi})=\xi_{2}^{0}.
\end{align} 
Consequently, recalling \eqref{phi-tilda}, we have that 
\begin{align}
    \widetilde{\phi}^{1}(\alpha_{1},\alpha_{2})=\alpha_{1}\qquad \text{and}\qquad \widetilde{\phi}^{2}(\alpha_{1},\alpha_{2})=\alpha_{2}\,.
\end{align}
As a consequence, we obtain
\begin{align}
     &\uplim_{k\to \infty}\sup_{g\;:\;D_{k}^{1}(g)=0\,,D_{k}^{2}(g)=0,\,\Delta_{k}(g)=0}\sum_{\bm{\zeta},\bm{\xi}\in \Omega_{2k+1}}\left|\left(\frac{1}{2k+1}\sum_{x=- k}^{k}\zeta_{1}^{x}\right)-\left(\frac{1}{2k+1}\sum_{x=- k}^{k}\xi_{1}^{x}\right)\right|g(\bm{\zeta},\bm{\xi})3^{-4k-2}\nonumber
     \\&\qquad +\uplim_{k\to \infty}
      \sup_{g\;:\;D_{k}^{1}(g)=0\,,D_{k}^{2}(g)=0,\,\Delta_{k}(g)=0}\sum_{\bm{\zeta},\bm{\xi}\in \Omega_{2k+1}}\left|\left(\frac{1}{2k+1}\sum_{x=-k}^{k}\zeta_{2}^{x}\right)-\left(\frac{1}{2k+1}\sum_{x=-k}^{k}\xi_{2}^{x}\right)\right|g(\bm{\zeta},\bm{\xi})3^{-4k-2}\nonumber
      \\ &\leq 
      \uplim_{k\to \infty}\sup_{0\leq k_{1},k_{2}\leq 4k+2\;:\;k_{1}+k_{2}\leq 4k+2}\sum_{\bm{\zeta},\bm{\xi}\in \Omega_{2k+1}}\frac{\left|\left(\frac{1}{2k+1}\sum_{x=- k}^{k}\zeta_{1}^{x}\right)-\widetilde{\phi}^{1}\left(\frac{k_{1}}{4k+2},\frac{k_{2}}{4k+2}\right)\right|}{\binom{4k+2}{k_{1},k_{2}}}\nonumber
      \\ &\qquad +
     \uplim_{k\to \infty} \sup_{0\leq k_{1},k_{2}\leq 4k+2\;:\;k_{1}+k_{2}\leq 4k+2}\sum_{\bm{\zeta},\bm{\xi}\in \Omega_{2k+1}}\frac{\left|\left(\frac{1}{2k+1}\sum_{x=- k}^{k}\xi_{1}^{x}\right)-\widetilde{\phi}^{1}\left(\frac{k_{1}}{4k+2},\frac{k_{2}}{4k+2}\right)\right|}{\binom{4k+2}{k_{1},k_{2}}}\nonumber
      \\&\qquad+
    \uplim_{k\to \infty}  \sup_{0\leq k_{1},k_{2}\leq 4k+2\;:\;k_{1}+k_{2}\leq 4k+2}\sum_{\bm{\zeta},\bm{\xi}\in \Omega_{2k+1}}\frac{\left|\left(\frac{1}{2k+1}\sum_{x=- k}^{k}\zeta_{2}^{x}\right)-\widetilde{\phi}^{2}\left(\frac{k_{2}}{4k+2},\frac{k_{2}}{4k+2}\right)\right|}{\binom{4k+2}{k_{1},k_{2}}}\nonumber
      \\&\qquad+
      \uplim_{k\to \infty}\sup_{0\leq k_{1},k_{2}\leq 4k+2\;:\;k_{1}+k_{2}\leq 4k+2}\sum_{\bm{\zeta},\bm{\xi}\in \Omega_{2k+1}}\frac{\left|\left(\frac{1}{2k+1}\sum_{x=- k}^{k}\xi_{2}^{x}\right)-\widetilde{\phi}^{2}\left(\frac{k_{2}}{4k+2},\frac{k_{2}}{4k+2}\right)\right|}{\binom{4k+2}{k_{1},k_{2}}}\nonumber=0\,.
\end{align}
The last equality follows from Corollary \ref{corollary-EqEns}.
\end{document}